\documentclass[review]{elsarticle}

\usepackage{lineno,hyperref}
\usepackage{amsfonts}
\usepackage{amsmath}
\modulolinenumbers[5]

\usepackage{xcolor}         
\usepackage{algorithm}
\usepackage{algorithmic}

\journal{TBD}

\newcommand{\calA}{{\mathcal{A}}}
\newcommand{\calB}{{\mathcal{B}}}
\newcommand{\calC}{{\mathcal{C}}}

\newcommand{\calF}{{\mathcal{F}}}

\newcommand{\calL}{{\mathcal{L}}}
\newcommand{\calM}{{\mathcal{M}}}
\newcommand{\calN}{{\mathcal{N}}}

\newcommand{\calP}{{\mathcal{P}}}

\newcommand{\calR}{{\mathcal{R}}}
\newcommand{\calS}{{\mathcal{S}}}
\newcommand{\calT}{{\mathcal{T}}}

\newcommand{\N}{{\mathbb{N}}} 
\newcommand{\R}{{\mathbb{R}}} 

\newtheorem{theorem}{Theorem} 
\newtheorem{lemma}[theorem]{Lemma}
\newtheorem{corollary}[theorem]{Corollary}
\newdefinition{definition}{Definition}
\newproof{proof}{Proof}

\begin{document}

\begin{frontmatter}

\title{Properties of functions on a bounded charge space}

\author{Jonathan M. Keith}
\address{School of Mathematics, Monash University, Wellington Rd, Clayton VIC 3800, Australia}

\begin{abstract}
A {\em charge space} $(X,\calA,\mu)$ is a generalisation of a measure space, consisting of a sample space $X$, a field of subsets $\calA$ and a finitely additive measure $\mu$, also known as a {\em charge}. Key properties a real-valued function on $X$ may possess include $T_1$-measurability and integrability. These properties are generalisations of corresponding properties of real-valued functions on a (countably additive) measure space. However, these properties are less well studied than their measure-theoretic counterparts.

This paper describes new characterisations of $T_1$-measurability and integrability in the case that the charge space is bounded, that is, $\mu(X) < \infty$. These characterisations are convenient for analytic purposes; for example, they facilitate simple proofs that $T_1$-measurability is equivalent to conventional measurability and integrability is equivalent to Lebesgue integrability, if $(X,\calA,\mu)$ is a complete measure space.

Several additional contributions to the theory of bounded charges are also presented. New characterisations of equality almost everywhere of two real-valued functions on a bounded charge space are provided. Necessary and sufficient conditions for the function space $L_1(X,\calA,\mu)$ to be a Banach space are determined. Lastly, the concept of completion of a measure space is generalised for charge spaces, and it is shown that under certain conditions, completion of a charge space adds no new equivalence classes to the quotient space $\calL_p(X,\calA,\mu)$.
\end{abstract}


\end{frontmatter}


\section{Introduction} \label{intro}

A {\em field of sets} $\calA$, also known as a {\em field of subsets} or an {\em algebra of sets}, is a subset of the power set $\calP(X)$ of some {\em sample space} $X$ such that:
\begin{enumerate}
\item $\emptyset \in \calA$,
\item $A^c \in \calA$ for all $A \in \calA$, and
\item $A \cup B \in \calA$ for all $A, B \in \calA$. 
\end{enumerate}
A $charge$ $\mu$ defined on $\calA$ is a real-valued function such that:
\begin{enumerate}
\item $\mu(\emptyset) = 0$, and
\item $\mu(A \cup B) = \mu(A) + \mu(B)$ for all disjoint $A, B \in \calA$.
\end{enumerate}
A charge that takes only non-negative values is sometimes called a {\em content} or {\em finitely additive measure}. However, in this paper the term {\em charge} is preferred, and may be assumed to refer to a non-negative valued charge unless otherwise stated. A triple $(X,\calA,\mu)$ of such objects is called a {\em charge space}. 

There is a little used but well developed theory of integration for real-valued functions on charge spaces. This theory actually predates the Lebesgue integral, beginning with the finitely additive measure theory developed by Giuseppe Peano (1858-1932) and independently by Camille Jordan (1838-1922) (see \cite{greco2010peano} for a nice historical account). The Lebesgue integral based on countably additive measures has proved to be more convenient and popular, but integration based on charges remains substantially more general. A brief review of this theory is provided in Section~\ref{theory_of_charges}, including definitions of the function spaces $L_p(X,\calA,\mu)$, comprised of integrable functions on a charge space. 

One reason finitely additive measure and integration theory have not become as popular as their countably additive counterparts is that definitions of the key properties of $T_1$-measurability and integrability are a little more complicated than corresponding properties in the countably additive context, and initially seem less convenient for analysis. In particular, conventionally measurable functions are pointwise approximable by a non-decreasing sequence of simple functions. It is not immediately clear whether this useful property holds in general for $T_1$-measurable functions. 

A goal of this paper is to facilitate analysis of functions on charge spaces, by providing more convenient characterisations of $T_1$-measurability and integrability. In Section~\ref{characterisations}, two new characterisations are developed for each of these properties, for the special case of a bounded charge space, that is, one for which $\mu(X) < \infty$. First, $T_1$-measurability is characterised in terms of inverse images of rays, using a criterion resembling a frequently used characterisation of conventional measurability, and an additional smoothness condition on the tails of the function. A characterisation of integrability is obtained by strengthening this smoothness condition. Second, $T_1$-measurability can be characterised in terms of pointwise approximation by an increasing sequence of simple functions, and integrability can be characterised by imposing an additional condition on this sequence. Section~\ref{characterisations} also contains two useful new characterisations of equality almost everywhere of functions on a charge space. 

To demonstrate the utility of the new characterisations, a number of simple corollaries are also proved in Section~\ref{characterisations}. In particular, it is shown that the $L_p$ spaces defined in Section~\ref{theory_of_charges} are equivalent to conventional Lebesgue function spaces for charge spaces that are also complete measure spaces. 

Another factor inhibiting the analysis of functions on charge spaces is that $L_p$ spaces are not necessarily complete in this context, unlike Lebesgue function spaces. To partially address this deficiency, Section~\ref{completeness_section} identifies necessary and sufficient conditions for $L_p$ spaces to be Banach spaces. 

The concept of a complete measure space is generalised for charges in Section~\ref{complete_fields}. This section also introduces a construction that is here called {\em null modification} (see Lemma~\ref{nullmodification_X} and proof). Null modification is used to show that, under certain condtions, expanding a field by adding additional null sets leaves the corresponding $\calL_p$ spaces unchanged up to isomorphism.


\section{Integration theory for charges}
\label{theory_of_charges}

As some readers may be unfamiliar with integration with respect to a charge, this section reviews the relevant theory based on the concise summary given in~\cite{basile2000} and the more comprehensive treatment in~\cite{bhaskararao1983}. Theorems and lemmas in this section that are stated without proof are proved in the references, but a few simple proofs are provided for results that are known but not explicitly proved in the literature. For simplicity, all charges and charge spaces mentioned in this paper are assumed to be non-negative and bounded, unless otherwise stated. 


\begin{definition}
A charge $\mu$ on a field of subsets $\calA$ of a set $X$ can be extended to an {\em outer charge} on $\calP(X)$ as follows
\[
\mu^*(A) := \inf\{ \mu(B): B \in \calA, A \subseteq B \}
\]
for all $A \in \calP(X)$. 
\end{definition}

Outer charges are sub-additive, that is, $\mu^*(A \cup B) \leq \mu^*(A) + \mu^*(B)$ for all $A,B \subseteq X$ (Proposition 4.1.4 of \cite{bhaskararao1983}).

Note the outer charge depends on the field of sets $\calA$. It is therefore sometimes useful to write $\mu^*_{\calA}$ to avoid ambiguity when working with nested fields. Nevertheless, the more concise notation $\mu^*$ will be used when the field of sets $\calA$ is clear from the context. 

\begin{definition} \label{null_function}
Let $(X,\calA,\mu)$ be a charge space. A {\em null set} is any $A \in \calP(X)$ for which $\mu^*(A) = 0$. A {\em null function} is a function $f : X \rightarrow \R$ such that
\[
\mu^*(\{ x \in X : | f(x) | > \epsilon \}) = 0
\]
for all $\epsilon > 0$.
\end{definition}

The following notion of equivalence between functions plays an important role throughout this paper.

\begin{definition} \label{equal_ae}
Let $(X,\calA,\mu)$ be a charge space. Two functions $f,g: X \rightarrow \R$ are said to be {\em equal almost everywhere} (abbreviated as $f = g$ a.e.) if $f - g$ is a null function.
\end{definition}

This terminology can be somewhat misleading. For example, consider the charge space $(\N, \calA, \nu)$, where $\N$ denotes the natural numbers (excluding zero), $\calA \subset \calP(X)$ is the field of sets comprised of the finite subsets of $\N$ and their complements (cofinite sets), and $\nu$ is the charge obtained by setting $\nu(A) = 0$ if $A \subset \N$ is finite and $\nu(A) = 1$ if $A$ is cofinite. The functions $f(n) = 1/n$ and $g(n) = 0$ are equal a.e., since $\{ x \in \N : | f(x) - g(x) | > \epsilon \}$ is finite, and therefore has zero charge for any $\epsilon$, even though $f(n) \neq g(n)$ for any $n \in \N$. 

A related definition is the following.

\begin{definition} \label{dominated_ae}
Let $(X,\calA,\mu)$ be a charge space. A function $f: X \rightarrow \R$ is said to be {\it dominated almost everywhere} by $g: X \rightarrow \R$ (abbreviated as $f \leq g$ a.e.)  if $f \leq g + h$, where $h$ is a null function. 
\end{definition}

If $\calA$ is a $\sigma$-field and $\mu$ a measure, then $f : X \rightarrow \R$ is a null function if and only if $\mu^*(\{ x : f(x) \neq 0 \}) = 0$ (Proposition 4.2.7 of \cite{bhaskararao1983}). In that case, the use of the term ``almost everywhere'' in the preceding two definitions corresponds to the conventional sense of this expression in measure theory. 

Whether a function $f$ is a null function, and whether $f$ is equal to or dominated almost everywhere by another function $g$, depends on the field $\calA$ via the outer charge $\mu^*_{\calA}$. Thus where there is potential ambiguity regarding the field, it is helpful to say that $f$ is a null function {\em with respect to} $\calA$, or that $f = g$ a.e. or $f \leq g$ a.e. {\em with respect to} $\calA$. Similar terminology should be used for all definitions that depend on an outer charge.

The following mode of convergence generalises convergence in probability and plays an important role in the definition of integrals over a charge space.

\begin{definition}
The mode of convergence $f_n \xrightarrow{h} f$ is read as ``$f_n$ converges to $f$ {\em hazily}'' and means that for every $\epsilon > 0$,
\[
\mu^*(\{ x: | f_n(x) - f(x) | > \epsilon \}) \rightarrow 0
\]
as $n \rightarrow \infty$.
\end{definition}

The limit of a sequence that converges hazily is not in general unique, but nevertheless the following result, adapted from Proposition~4.3.2 of \cite{bhaskararao1983}, holds.

\begin{theorem} \label{hazy_uniqueness}
Let $(X,\calA,\mu)$ be a charge space and let $f$ and $g$ be real-valued functions on $X$. If a sequence of real-valued functions $\{ f_k \}_{k=1}^{\infty}$ on $X$ converges hazily to $f$, and $f = g$ a.e., then $\{ f_k \}_{k=1}^{\infty}$ converges to $g$ hazily. Conversely, if $\{ f_k \}_{k=1}^{\infty}$ converges to both $f$ and $g$ hazily, then $f = g$ a.e.
\end{theorem}

For two fields of subsets $\calA, \calA^{\prime}$ of a set $X$ with $\calA^{\prime} \subset \calA$ and a charge $\mu$ defined on  $\calA$, ambiguity can arise as to whether hazy convergence is with respect to $\calA$ or $\calA^{\prime}$. In what follows, the notation $f_n \xrightarrow{h,\calA} f$ is used wherever there is potential ambiguity regarding the field of sets with respect to which hazy convergence occurs. 

The following lemma confirms an expected property of dominating functions: that the limit of a dominated sequence is also dominated.

\begin{lemma} \label{limit_dominated}
Let $(X,\calA,\mu)$ be a charge space, and consider a sequence of real-valued functions $\{ f_k \}_{k=1}^{\infty}$ on $X$ and real-valued functions $f$ and $g$ on $X$ such that $f_k \xrightarrow{h} f$ and $f_k \leq g$ a.e. for all $k \in \N$. Then $f \leq g$ a.e.
\end{lemma}

\begin{proof}
For each $k \in \N$, there exists a null function $h_k$ such that $f_k \leq g + h_k$. Without loss of generality, suppose $h_k \geq 0$ (noting $f_k \leq g +h_k^+$, where $h_k^+ := \max \{ 0, h_k \}$). Then
\[
f \leq f_k + | f_k - f | \leq g + h_k + | f_k - f |,
\]
and thus $f \leq g + h$, where $h := \inf_k \left( h_k + | f_k - f | \right) \geq 0$. Now for any $x \in X$ and $\epsilon > 0$,
\begin{eqnarray*}
h(x) > \epsilon &\implies& h_k(x) + | f_k - f | > \epsilon, \mbox{ for all } k \in \N \\
&\implies& h_k(x) > \frac{\epsilon}{2} \mbox{ or } | f_k - f | > \frac{\epsilon}{2}, \mbox{ for all } k \in \N.
\end{eqnarray*}
Hence for each $k \in \N$,
\[
\mu^*(\{ x : h(x) > \epsilon \}) \leq \mu^*(\{ x : h_k(x) > \frac{\epsilon}{2} \}) + \mu^*(\{ x : | f_k - f | > \frac{\epsilon}{2} \}).
\]
The first term on the right hand side is zero, and the second goes to zero as $k \rightarrow \infty$, implying $\mu^*(\{ x : h(x) > \epsilon \}) = 0$ for any $\epsilon > 0$. That is, $h$ is a null function.
\qed \end{proof}

As in standard Lebesgue integration, the integral is defined with reference to simple functions.

\begin{definition}
A {\em simple function} on a charge space $(X,\calA,\mu)$ is a function $f: X \rightarrow \R$ of the form $\sum_{k=1}^K c_k I_{A_k}$, for $K \in \N$, real numbers $c_1, \ldots, c_K$ and a partition of $X$ into subsets $\{ A_1, \ldots, A_K \} \subset \calA$. (For unbounded charges, simple functions must also satisfy $\mu(A_k) < \infty$ whenever $c_k \neq 0$.)
\end{definition}

Simple functions are used to define a notion of measurability of functions on a charge space.

\begin{definition}
A function $f: X \rightarrow \R$ is said to be {\em $T_1$-measurable} on a charge space $(X,\calA,\mu)$ if there is a sequence of simple functions $f_n \xrightarrow{h} f$. It is said to be {\em $T_2$-measurable} if, given any $\epsilon > 0$, there is a partition of $X$ into a finite number of sets $\{ A_0, A_1, \ldots, A_K \} \subseteq \calA$ such that $\mu(A_0) < \epsilon$ and $| f(x) -f(y) | < \epsilon$ for every $x, y \in A_k$ and $k \in \{ 1, \ldots, K \}$.
\end{definition}

These two definitions of measurability are equivalent (Theorem~4.4.7 of~\cite{bhaskararao1983}). 

\begin{definition} \label{L0}
The set of functions of the form $f : X \rightarrow \R$ that are $T_1$-measurable with respect to a charge space $(X,\calA,\mu)$ is denoted $L_0(X,\calA,\mu)$. 

A pseudo-metric can be defined for $L_0(X,\calA,\mu)$ as follows: 
\[
d(f,g) := \inf \{\epsilon > 0 : \mu^*(\{ x : | f(x) - g(x) | > \epsilon \}) < \epsilon \}.
\]
for all $f, g \in L_0(X,\calA,\mu)$. 
\end{definition}

It is sometimes helpful to write $d_{\calA}(f,g)$ to clarify the field of sets used to define the outer charge $\mu^*_{\calA}$. It is immediate that $d_{\calA}(f,g) = 0$ if and only if $f = g$ a.e. with respect to $\calA$, and that a sequence of functions $\{ f_n \}_{n=1}^{\infty} \subseteq L_0(X,\calA,\mu)$ converges to $f \in L_0(X,\calA,\mu)$ in the pseudo-metric $d_{\calA}$ if and only if $f_n \xrightarrow{h,\calA} f$.

Define $\calL_0(X,\calA,\mu)$ to be the collection of equivalence classes of $L_0(X,\calA,\mu)$ under the equivalence relation $f \sim g \iff d(f,g) = 0$. Then the function $d([f],[g]) := d(f,g)$ is a a metric for $\calL_0(X,\calA,\mu)$, where $[f]$ and $[g]$ are the equivalence classes of $f$ and $g$ respectively.

Any $T_1$-measurable function has the following property.

\begin{definition}
A function $f: X \rightarrow \R$ is said to be {\em smooth} if, given any $\epsilon > 0$, there exists $k > 0$ such that 

\[
\mu^*(\{ x \in X : | f(x) | > k \}) < \epsilon.
\]
\end{definition}

Although $T_1$-measurability implies smoothness (Corollary~4.4.8 of~\cite{bhaskararao1983}), the converse is not true in general. 

Other key properties of $T_1$-measurable functions include the following, adapted from Corollary~4.4.9 of \cite{bhaskararao1983}.

\begin{theorem} \label{T1_properties}
Let $(X,\calA,\mu)$ be a charge space and let $f$ and $g$ be real-valued functions on $X$. Let $\psi : \R \rightarrow \R$ be a continuous function. Then 
\begin{enumerate}
\item If $f$ and $g$ are $T_1$-measurable and $c,d \in \R$, then $cf+dg$ and $fg$ are $T_1$-measurable.
\item If $f$ is $T_1$-measurable, then $\psi(f)$ is $T_1$-measurable.
\item If $\{ f_n \}_{n=1}^{\infty}$ is a sequence of $T_1$-measurable functions converging to $f$ hazily, then $f$ is $T_1$-measurable and $\{ \psi(f_n) \}_{n=1}^{\infty}$ converges to $\psi(f)$ hazily.
\end{enumerate}
\end{theorem}

\begin{definition} \label{integrable}
Any simple function is {\em integrable} with integral
\[
\int \sum_{k=1}^K c_k I_{A_k} := \sum_{k=1}^K c_k \mu(A_k).
\]
A more general function $f: X \rightarrow \R$ is said to be {\em integrable} if there is a sequence of simple functions $\{ f_n \}_{n=1}^{\infty}$ such that:
\begin{enumerate}
\item $f_n \xrightarrow{h} f$, and
\item $\int | f_n - f_m | d\mu \rightarrow 0$ as $n,m \rightarrow \infty$.
\end{enumerate}
Such a sequence is said to be a {\em determining sequence} for $f$. The integral is then given by 
\[
\int f d\mu := \lim_{n \rightarrow \infty} \int f_n d\mu.
\]
This integral is well defined, that is the limit has the same value for any determining sequence (Theorem 4.4.10 in \cite{bhaskararao1983}).
\end{definition}

For two fields of sets $\calA, \calA^{\prime}$ with $\calA^{\prime} \subset \calA$ and a charge $\mu$ defined on $\calA$, ambiguity can arise as to whether a function is integrable with respect to $\calA$ or $\calA^{\prime}$. Where necessary to distinguish the two integrals, they are denoted $\int f d\mu_{\calA}$ and $\int f d\mu_{\calA^{\prime}}$ respectively. Where the relevant field of sets is clear from the context, the subscript can be omitted. 

The following theorem, adapted from Theorem~4.4.13 of~\cite{bhaskararao1983}, enumerates some key properties of integrable functions.

\begin{theorem} \label{L1_integrable_properties}
Let $(X,\calA,\mu)$ be a charge space and let $f$ and $g$ be real-valued functions on $X$. Then 
\begin{enumerate}
\item If $f$ and $g$ are integrable and $c$ and $d$ are real numbers, then $cf + dg$ is integrable and $\int (cf + dg) d\mu = c \int f d\mu + d \int g d\mu$.
\item $f$ is integrable if and only if $f^+ := \max \{ f, 0 \}$ and $f^- := \max\{ -f, 0 \}$ are integrable. 
\item If $f$ and $g$ are integrable and $f \leq g$ a.e., then $\int f d\mu \leq \int g d\mu$.
\item If $f$ is integrable and $g = f$ a.e. then $g$ is integrable and $\int g d\mu = \int f d\mu$.
\end{enumerate}
\end{theorem}

Another important property of integrals with respect to charges is given by Theorem~4.4.18 of~\cite{bhaskararao1983}, which may be expressed as follows.

\begin{theorem} \label{dominated_integrability}
Let $(X,\calA,\mu)$ be a charge space and let $f$ and $g$ be real-valued functions on $X$ such that $| g | \leq f$ a.e. and $f$ is integrable. Then $g$ is integrable if and only if it is $T_1$-measurable.
\end{theorem}

\begin{definition} \label{Lp}
For $p \in [1,\infty)$, the function space $L_p(X,\calA,\mu)$ is the set of all $T_1$-measurable functions $f: X \rightarrow \R$ such that $| f |^p$ is integrable.
\end{definition}

The following theorem paraphrases Theorem~4.6.7 of~\cite[Sec. 4.6]{bhaskararao1983}).

\begin{theorem} \label{Lp_vector_space}
Let $(X,\calA,\mu)$ be a charge space and $p \in [1,\infty)$. Then $L_p(X,\calA,\mu)$ is a vector space with a pseudonorm
\[
\| f \|_p = \left( \int | f |^p d\mu \right)^{1/p}
\]
\end{theorem}

One can construct a normed vector space from $L_p(X,\calA,\nu)$ as follows. 

\begin{definition}
For $p \in [1,\infty)$, the function space $\calL_p(X,\calA,\mu)$ is the set of equivalence classes in $L_p(X,\calA,\mu)$ under the equivalence relation $f \sim g$ if and only if $\| f - g \|_p = 0$. 
\end{definition}

The following characterisation of this equivalence relation is a consequence of Theorem~4.4.13(ix) of~\cite{bhaskararao1983} (see also Comment~1.5 of \cite{basile2000}).

\begin{theorem} \label{Lp_equivalence_classes}
For $p \in \{0\} \cup [1,\infty)$ and $f, g \in L_p(X,\calA,\mu)$, $f \sim g$ if and only if $f$ and $g$ are equal a.e.  
\end{theorem}

Note the functional $\| [f] \|_p = \| f \|_p$ is a norm for $\calL_p(X,\calA,\mu)$, where $[f]$ is the equivalence class of $f \in L_p(X,\calA,\mu)$. Where there is potential ambiguity or uncertainty regarding the field $\calA$ with respect to which the equivalence class is defined, the notation $[f]_{\calA}$ will be used to clarify.

It is also possible to define pseudometric spaces $L_p(X,\calA,\mu)$ and metric spaces $\calL_p(X,\calA,\mu)$ for $p \in (0,1)$ (see \cite{basile2000}). Much of the theory developed in this paper could be generalised to include these spaces, but as these spaces are of limited interest, and for the sake of simplicity, the paper focuses on the cases $p = 0$ and $p \in [1,\infty)$.

The following simple lemma is frequently useful. 

\begin{lemma} \label{Lp_pos_neg_parts}
Let $(X,\calA,\mu)$ be a charge space and $p \in [1,\infty)$. Then the following three statements are logically equivalent.
\begin{enumerate}
\item $f \in L_p(X,\calA,\mu)$.
\item $(f^+)^p, (f^-)^p \in L_1(X,\calA,\mu)$.
\item $f^+, f^- \in L_p(X,\calA,\mu)$.
\end{enumerate}
Moreover, if any of these statements holds, then $| f | \in L_p(X,\calA,\mu)$ and $| f |^p \in L_1(X,\calA,\mu)$.
\end{lemma}

\begin{proof}
$(1 \implies 2)$ $(f^+)^p$ and $(f^-)^p$ are $T_1$-measurable by Theorem~\ref{T1_properties}(2), and then integrable by Theorem~\ref{dominated_integrability}.

$(2 \implies 3)$ This is immediate from the definition of $L_p(X,\calA,\mu)$.

$(3 \implies 1)$ This follows by Theorem~\ref{Lp_vector_space}, since $f = f^+ - f^-$.

The additional statements follow by Theorem~\ref{T1_properties} and the definition of $L_p(X,\calA,\mu)$.
\qed \end{proof}

The $L_p$ spaces respect the partial ordering on fields of subsets in the sense determined by the following lemma, which follows straightforwardly from the relevant definitions (the proof is omitted).

\begin{lemma} \label{nested_fields}
Suppose $(X,\calA,\mu)$ and $(X,\calA^{\prime},\mu)$ are charge spaces with $\calA^{\prime} \subseteq \calA$. Then the following statements hold.
\begin{enumerate}
\item If $f_n \xrightarrow{h,\calA^{\prime}} f$ for real-valued functions $\{ f_n \}_{n=1}^{\infty}$ and $f$ on $X$, then $f_n \xrightarrow{h,\calA} f$.
\item If $f$ is a simple function with respect to $\calA^{\prime}$, then it is simple with respect to $\calA$.
\item If $f$ is a $T_1$-measurable function with respect to $\calA^{\prime}$, then it is $T_1$-measurable with respect to $\calA$.
\item If $f$ is an integrable function with respect to $\calA^{\prime}$, then it is integrable with respect to $\calA$. Moreover, $\int f d\mu_{\calA^{\prime}} = \int f d\mu_{\calA}$.
\item $L_p(X,\calA^{\prime},\mu) \subseteq L_p(X,\calA,\mu)$ for $p \in \{0\} \cup [1,\infty)$.
\item $\mu^*_{\calA^{\prime}}(A) \geq \mu^*_{\calA}(A)$ for all $A \in \calA$.
\end{enumerate}
\end{lemma}

The $L_p$ spaces are nested as described in the following theorem (Corollary~4.6.5 of~\cite{bhaskararao1983}).

\begin{theorem} \label{Lp_nesting}
Let $(X,\calA,\mu)$ be a charge space and $r, s \in [1,\infty)$ with $r \leq s$. Then
\[
L_1(X,\calA,\mu) \supseteq L_r(X,\calA,\mu) \supseteq L_s(X,\calA,\mu).
\]
\end{theorem}

Like the Lebesgue integral, integration with respect to a charge satisfies a {\em dominated convergence} theorem. The following is adapted from Theorem~4.6.14 of~\cite{bhaskararao1983}.

\begin{theorem} \label{dominated_convergence}
Let $(X,\calA,\mu)$ be a charge space and let $g \in L_p(X,\calA,\mu)$ for some $p \in [1,\infty)$. Let $\{ f_k \}_{k=1}^{\infty}$ be a sequence of $T_1$-measurable functions on $X$ such that $| f_k | \leq g$ a.e. for each $k \in \N$, and let $f$ be a real-valued function on $X$. Then $f_k \xrightarrow{h} f$ if and only if $f \in L_p(X,\calA,\mu)$ and $\| f_k - f \|_p \rightarrow 0$.
\end{theorem}

The simple functions in $L_p(X,\calA,\mu)$ have the following property, adapted from Theorem~4.6.15 of~\cite{bhaskararao1983}.

\begin{theorem} \label{simple_dense}
Let $(X,\calA,\mu)$ be a charge space. The simple functions on $X$ with respect to $\calA$ are dense in $L_p(X,\calA,\mu)$ for every $p \in \{0\} \cup [1,\infty)$.
\end{theorem}

This is proved for the case $p \in [1,\infty)$ in \cite{bhaskararao1983}. The case $p = 0$ follows from the fact that a sequence of functions $\{f_n\}_{n=1}^{\infty}$ converges hazily to $f$ if and only if $d(f_n,f) \rightarrow 0$ (see Section~1.1 of \cite{basile2000}).

An important tool in the analysis of $L^p$ spaces with respect to a charge is the Peano-Jordan completion (\cite{basile2000, gangopadhyay1990,greco1982}), defined as follows.

\begin{definition} \label{PJ_completion}
Let $(X,\calA,\mu)$ be a charge space. The {\em Peano-Jordan completion} of $\calA$ is the charge space $(X,\overline{\calA},\overline{\mu})$ where
\[
\overline{\calA} := \{ A \subseteq X : \forall \epsilon > 0, \exists B, C \in \calA \mbox{ such that } B \subseteq A \subseteq C \mbox{ and } \mu(C \setminus B) < \epsilon \}
\]
and
\[
\overline{\mu}(A) := \sup \{\mu(B): B \subseteq A, B \in \calA\} = \inf \{\mu(C): A \subseteq C, C \in \calA\}.
\]
A charge space is said to be {\em Peano-Jordan} complete if it is equal to its Peano-Jordan completion. 
\end{definition}

It is straightforward to show $\overline{\calA}$ is a field and $\overline{\mu}$ is a charge. One can also show Peano-Jordan completion is an idempotent operation, that is $\overline{\overline{\calA}} = \overline{\calA}$ and $\overline{\overline{\mu}} = \overline{\mu}$. Two more useful properties are that $\mu_{\calA}^*(A) = \overline{\mu}(A)$ for $A \in \overline{\calA}$, and that $(X,\overline{\calA},\overline{\mu})$ and $(X,\calA,\mu)$ have the same null sets.

If $(X,\calA,\mu)$ and $(X,\calA^{\prime},\mu)$ are charge spaces with $\calA^{\prime} \subseteq \calA$, and $(X,\overline{\calA},\overline{\mu}_{\calA})$ and $(X,\overline{\calA^{\prime}},\overline{\mu}_{\calA^{\prime}})$ are their respective Peano-Jordan completions, then it is straightforward to show $\overline{\calA^{\prime}} \subseteq \overline{\calA}$ and $\overline{\mu}_{\calA^{\prime}}(A) = \overline{\mu}_{\calA}(A)$ for all $A \in \calA^{\prime}$. Thus the subscripts can safely be omitted from $\overline{\mu}_{\calA}$ and $\overline{\mu}_{\calA^{\prime}}$ without ambiguity. However, it is useful to retain the subscripts on the inverse set functions $\overline{\mu}_{\calA}^{-1}$ and $\overline{\mu}_{\calA^{\prime}}^{-1}$, since these are not in general equal. Note also that $\mu^*_{\calA^{\prime}}(A) = \mu^*_{\calA}(A) = \overline{\mu}(A)$ for $A \in \overline{\calA^{\prime}}$, and the subscripts may therefore be omitted without ambiguity in this case as well.

The following lemma establishes that Peano-Jordan completeness generalises the concept of completeness for measure spaces.

\begin{lemma} \label{equivalent_completeness}
Suppose $(X,\calA,\mu)$ is a measure space. Then it is Peano-Jordan complete if and only if it is a complete measure space.
\end{lemma}

\begin{proof}
$(\implies)$ Let $N \in \calA$ such that $\mu(N) = 0$, and consider $A \subset N$. Then $\emptyset \subseteq A \subseteq N$ and $\mu(N \setminus \emptyset) = 0$, implying $A \in \overline{\calA} = \calA$. 

$(\impliedby)$ Consider $A \in \overline{\calA}$. Then there exist $B,C \in \calA$ such that $B \subseteq A \subseteq C$ and $\mu(C \setminus B) = 0$. (To see this, use the definition of $\overline{\calA}$ to construct $B_n,C_n \in \calA$ with $B_n \subseteq A \subseteq C_n$ and $\mu(C_n \setminus B_n) < 1/n$, then set $B := \cup_{i=1}^{\infty} B_n$ and $C := \cap_{i=1}^{\infty} C_n$, noting $\calA$ is closed under countable unions and intersections.) Thus $A \in \calA$, since the measure space is complete. 
\qed \end{proof}



Peano-Jordan completion of a charge space leaves its $L_p$ spaces unchanged, as the following theorem, adapted from Proposition~1.8 of~\cite{basile2000}, states.

\begin{theorem} \label{prop_1.8_basile2000}
Let $(X,\calA,\mu)$ be a charge space and let $(X,\overline{\calA},\overline{\mu})$ be its Peano-Jordan completion. Then 
\begin{enumerate}
\item For $A \subseteq X$, $\mu^*(A) = \overline{\mu}^*(A)$.
\item $A \in \overline{\calA}$ if and only if $I_A \in L_0(X,\calA,\mu)$.
\item $L_0(X,\calA,\mu) = L_0(X, \overline{\calA}, \overline{\mu})$.
\item $L_p(X,\calA,\mu) = L_p(X, \overline{\calA}, \overline{\mu})$ for all $p \in [1,\infty)$.
\item For $f \in L_1(X,\calA,\mu)$, $\int f d\mu = \int f d\overline{\mu}$.
\end{enumerate}
\end{theorem}

The above theorem has the following consequence.

\begin{corollary} \label{fI_A}
Let $(X,\calA,\mu)$ be a charge space and $p \in \{0\} \cup [1,\infty)$. Then $f \in L_p(X,\calA,\mu)$ if and only if $f I_A \in L_p(X,\calA,\mu)$ for all $A \in \overline{\calA}$. 
\end{corollary}

\begin{proof}
Suppose $f \in L_p(X,\calA,\mu)$ and $A \in \overline{\calA}$. By Theorem~\ref{prop_1.8_basile2000}, $I_A$ is $T_1$-measurable, and by Theorem~\ref{T1_properties}, $f I_A$ is $T_1$-measurable. This is all that is required for the case $p=0$. For $p \in [1,\infty)$, Theorem~\ref{dominated_integrability} gives $| f I_A |^p \in L_1(X,\calA,\mu)$, that is, $f I_A \in L_p(X,\calA,\mu)$.

The converse is trivial, since $X \in \overline{\calA}$.
\qed \end{proof}

One final theorem, useful for showing that an $L_p$ space is complete, is the following, adapted from Theorem~3.4 of~\cite{basile2000}.

\begin{theorem} \label{Lp_complete}
Let $(X,\calA,\mu)$ be a charge space and $p \in [1,\infty)$. The following statements are logically equivalent.
\begin{enumerate}
\item $L_p(X,\calA,\mu)$ is complete.
\item $L_0(X,\calA,\mu)$ is complete.
\item If $A_1 \subseteq A_2 \subseteq \ldots$ are sets in $\calA$ such that $\sup_{k \geq 1} \mu(A_k) < \infty$ and $\epsilon > 0$, there exists $D \in \calA$ such that $\mu(A_k \setminus D) = 0$ for all $k \geq 1$ and $\lim_{k \rightarrow \infty} \mu(A_k) \leq \mu(D) \leq \lim_{k \rightarrow \infty} \mu(A_k) + \epsilon$.
\end{enumerate}
\end{theorem}

\section{$T_1$-measurability, integrability, and equality almost everywhere}
\label{characterisations}

This section presents a contribution to the general theory of bounded charges. New characterisations of $T_1$-measurability, integrability, and equality almost everywhere of functions on a charge space are deduced. Among other corollaries, it follows that $L_p(X,\calA,\mu)$ spaces defined as in Section~\ref{theory_of_charges} are equivalent to conventional Lebesgue function spaces when $(X,\calA,\mu)$ is a complete measure space (Corollaries~\ref{measurable_fn} and~\ref{integrable_fn}). The latter result is known, or at least generally assumed, but proof seems to be absent from the literature.

The following two lemmas are fundamental to the results in this section.

\begin{lemma} \label{atomimpliesrayinverseimage}
Consider $f \in L_0(X,\calA,\mu)$, where $(X,\calA,\mu)$ is a charge space. Then for all $y \in \R$ such that $\mu^*(f^{-1}[y-\delta,y+\delta]) \rightarrow 0$ as $\delta \rightarrow 0$, $f^{-1}(-\infty,y] \in \overline{\calA}$ and $f^{-1}[y,\infty) \in \overline{\calA}$. Moreover, $f^{-1}(y) \in \overline{\calA}$ with $\overline{\mu}(f^{-1}(y)) = 0$.
\end{lemma}

\begin{proof}
Suppose $y \in \R$ satisfies the condition of the lemma, fix $\epsilon > 0$ and select $\delta > 0$ and $D \in \calA$ such that $f^{-1}[y-\delta, y+\delta] \subseteq D$ with $\mu(D) < \epsilon/2$. Since $f$ is $T_2$-measurable, there exists a partition of $X$ into sets $\{ A_0, A_1, \ldots, A_K \} \subseteq \calA$ such that $\mu(A_0) < \epsilon/2$ and $|f(x_1) - f(x_2)| < 2\delta$ for every $x_1, x_2 \in A_k$ and $k \in \{ 1, \ldots, K \}$. For each $k \in \{ 1, \ldots, K \}$, exactly one of the following is true:
\begin{enumerate}
\item $A_k \subseteq f^{-1}[y-\delta,y+\delta]$, 
\item $A_k \cap f^{-1}(-\infty,y-\delta) \neq \emptyset$, or 
\item $A_k \cap f^{-1}(y+\delta,\infty) \neq \emptyset$.
\end{enumerate} 
Define 
\[
B := \bigcup \{ A_k : k \in \{ 1, \ldots, K \}, A_k \cap f^{-1}(-\infty,y-\delta) \neq \emptyset \} \setminus D
\]
and \[
C^c := \bigcup \{ A_k : k \in \{ 1, \ldots, K \}, A_k \cap f^{-1}(y+\delta,\infty) \neq \emptyset \} \setminus D.
\]
Then $B,C \in \calA$ with $B \subseteq f^{-1}(-\infty,y] \subseteq C$ and $C \setminus B = (B \cup C^c)^c \subseteq A_0 \cup D$. Hence $\mu(C \setminus B) \leq \mu(A_0) + \mu(D) < \epsilon$, and $f^{-1}(-\infty,y] \in \overline{\calA}$. The proof that $f^{-1}[y,\infty) \in \overline{\calA}$ is similar.

Fix $\epsilon > 0$ and choose $\delta > 0$ such that $\mu^*(f^{-1}[y-\delta,y+\delta]) < \epsilon$. Then $f^{-1}(y) = f^{-1}(-\infty,y] \cap f^{-1}[y,\infty) \in \overline{\calA}$ with $\overline{\mu}(f^{-1}(y)) \leq \mu^*(f^{-1}[y-\delta,y+\delta]) < \epsilon$. Let $\epsilon \rightarrow 0$ to obtain $\overline{\mu}(f^{-1}(y)) = 0$.
\qed \end{proof}

Note the proof of Lemma~\ref{atomimpliesrayinverseimage} does not require $\mu$ to be bounded.



\begin{lemma} \label{countableatoms}
Consider $f \in L_0(X,\calA,\mu)$, where $(X,\calA,\mu)$ is a charge space. There are at most countably many $y \in \R$ such that $\lim _{\delta \rightarrow 0} \mu^*(f^{-1}[y-\delta,y+\delta]) > 0$. 
\end{lemma}
\begin{proof}
For any $n \in \N$, suppose there exist distinct $\{ y_1, \ldots, y_n \} \subset \R$ such that 
\[
\lim _{\delta \rightarrow 0} \mu^*(f^{-1}[y_i-\delta,y_i+\delta]) > \frac{\mu(X)}{n}
\] 
for each $i \in \{1, \ldots, n \}$. Choose $\delta > 0$ with $\delta < \min \{ | y_i - y_j | : i,j \in \{1, \ldots, n \} \}$
and note $\mu^*(f^{-1}[y_i-\delta/4,y_i+\delta/4]) > \mu(X)/n$ for each $i$. Then choose $\epsilon > 0$ such that $\mu^*(f^{-1}[y_i-\delta/4,y_i+\delta/4]) > \mu(X)/n + \epsilon$ for each $i$.

Since $f$ is $T_2$-measurable, there exists a partition $\{ A_0, A_1, \ldots, A_K \} \subseteq \calA$ of $X$ such that $\mu(A_0) < \epsilon$ and $| f(x_1) - f(x_2) | < \delta/4$ for all $x_1, x_2 \in A_j$ and all $j \in \{ 1, \ldots, K \}$. For each $i \in \{1, \ldots, n \}$, define 
\[
B_i := \bigcup \{ A_j : k \in \{ 1, \ldots, K \}, A_j \cap f^{-1}[y_i-\delta/4,y_i+\delta/4] \neq \emptyset \}.
\]
Then $B_i \in \calA$ with $f^{-1}[y_i-\delta/4,y_i+\delta/4] \subseteq B_i \cup A_0$, and hence 
\[
\mu(B_i) \geq \mu^*(f^{-1}[y_i-\delta/4,y_i+\delta/4]) - \mu(A_0) > \mu(X)/n.
\]
Moreover, $B_i \subseteq f^{-1}[y_i-\delta/2, y_i+\delta/2]$ for each $i$, implying $\{ B_i \}_{i=1}^n$ is a collection of pairwise disjoint sets in $\calA$, and thus $\mu(\cup_{i=1}^n B_i) = \sum_{i=1}^n \mu(B_i) > \mu(X)$, a contradiction. It follows there can be at most finitely many (fewer than $n$) elements $y \in \R$ with $\lim _{\delta \rightarrow 0} \mu^*(f^{-1}[y-\delta, y+\delta]) > \mu(X)/n$, and hence at most countably many $y \in \R$ with $\lim _{\delta \rightarrow 0} \mu^*(f^{-1}[y-\delta, y+\delta]) > 0$.
\qed \end{proof}

Putting the above two lemmas together gives that $T_1$-measurable functions produce chains of inverse images in $\overline{\calA}$, after excluding a countable set. As the following lemma states, that countable set may be chosen to correspond precisely with the sets at which $\overline{\mu}$ is in a certain sense discontinuous on the chain.

\begin{lemma} \label{countableexceptions}
Consider $f \in L_0(X,\calA,\mu)$, where $(X,\calA,\mu)$ is a charge space. There exists a countable set $C \subset \R$ such that
\begin{enumerate}
\item $f^{-1}(y,\infty) \in \overline{\calA}$ for all $y \in \R \setminus C$, and
\item $y \in \R \setminus C$ if and only if 
\begin{eqnarray*}
\mu^*(f^{-1}(y,\infty)) &=& \sup \{ \overline{\mu}(A) : A \in \calT, f^{-1}(y,\infty) \supset A \} \\
&=& \inf \{ \overline{\mu}(A) : A \in \calT, f^{-1}(y,\infty) \subset A \}
\end{eqnarray*}
\end{enumerate}
where $\calT := \{ f^{-1}(y,\infty) : y \in \R \setminus C \}$. 
\end{lemma}

\begin{proof}
By Lemma~\ref{countableatoms}, the set
\[
C := \{ y \in \R : \lim _{\delta \rightarrow 0} \mu^*(f^{-1}[y-\delta,y+\delta]) > 0 \}
\]
is countable. By Lemma~\ref{atomimpliesrayinverseimage}, $f^{-1}(y,\infty) = (f^{-1}(-\infty,y])^c \in \overline{\calA}$, for all $y \in \R \setminus C$.

Suppose $y \in \R \setminus C$ and fix $\epsilon > 0$. Then there exists $\delta > 0$ such that $\mu^*(f^{-1}[y-\delta,y+\delta]) < \epsilon$, and $y^{\prime} \in \R \setminus C$ such that $y < y^{\prime} < y + \delta$. It follows that
\begin{eqnarray*}
\overline{\mu}(f^{-1}(y,\infty)) - \overline{\mu}(f^{-1}(y^{\prime},\infty)) & = & \overline{\mu}(f^{-1}(y,\infty) \setminus f^{-1}(y^{\prime},\infty)) \\
& = & \mu^*(f^{-1}(y,y^{\prime}]) \\
& \leq & \mu^*(f^{-1}[y-\delta,y+\delta]) \\
& < & \epsilon.
\end{eqnarray*}
But then
\begin{eqnarray*}
\sup \{ \overline{\mu}(A) : A \in \calT, f^{-1}(y,\infty) \supset A \} & \leq & \overline{\mu}(f^{-1}(y,\infty)) \\
& < & \overline{\mu}(f^{-1}(y^{\prime},\infty)) + \epsilon \\
& \leq & \sup \{ \overline{\mu}(A) : A \in \calT, f^{-1}(y,\infty) \supset A \} + \epsilon.
\end{eqnarray*}
Letting $\epsilon \rightarrow 0$ gives
\[
\mu^*(f^{-1}(y,\infty)) = \overline{\mu}(f^{-1}(y,\infty)) = \sup \{ \overline{\mu}(A) : A \in \calT, f^{-1}(y,\infty) \supset A \}.
\]
The proof that $\mu^*(f^{-1}(y,\infty)) = \inf \{ \overline{\mu}(A) : A \in \calT, f^{-1}(y,\infty) \subset A \}$ is similar.

Conversely, suppose 
\begin{eqnarray*}
\mu^*(f^{-1}(y,\infty)) &=& \sup \{ \overline{\mu}(A) : A \in \calT, f^{-1}(y,\infty) \supset A \} \\
&=& \inf \{ \overline{\mu}(A) : A \in \calT, f^{-1}(y,\infty) \subset A \}.
\end{eqnarray*}
Fix $\epsilon > 0$. Since $C$ is countable, one may choose $\delta > 0$ so that $y - \delta \in \R \setminus C$ and $y + \delta \in \R \setminus C$ with $\overline{\mu}(f^{-1}(y-\delta,\infty)) < \mu^*(f^{-1}(y,\infty)) + \epsilon/2$ and $\overline{\mu}(f^{-1}(y+\delta,\infty)) > \mu^*(f^{-1}(y,\infty)) - \epsilon/2$. Note $\overline{\mu}(f^{-1}[y-\delta,\infty)) = \overline{\mu}(f^{-1}(y-\delta,\infty))$, since $\overline{\mu}(f^{-1}(y-\delta)) = 0$ by Lemma~\ref{atomimpliesrayinverseimage}. Hence
\begin{eqnarray*}
\mu^*(f^{-1}[y-\delta,y+\delta]) & = & \overline{\mu}(f^{-1}[y-\delta,y+\delta]) \\
& = & \overline{\mu}(f^{-1}[y-\delta,\infty)) - \overline{\mu}(f^{-1}(y+\delta,\infty)) \\
& < & [\mu^*(f^{-1}(y,\infty)) + \epsilon/2] - [\mu^*(f^{-1}(y,\infty)) - \epsilon/2] \\
& = & \epsilon.
\end{eqnarray*}
Hence $\lim _{\delta \rightarrow 0} \mu^*(f^{-1}[y-\delta,y+\delta]) = 0$ and $y \in \R \setminus C$.
\qed \end{proof}

By similar reasoning, the preceding lemma also holds if one replaces $f^{-1}(y,\infty)$ with $f^{-1}(-\infty,y)$, $f^{-1}[y,\infty)$, or $f^{-1}(-\infty,y]$.

The condition on chains of inverse images identified in Lemma~\ref{countableexceptions}, together with smoothness, turns out to characterise $T_1$-measurable functions.

\begin{theorem} \label{T1measurable_characterisation}
Let $(X,\calA,\mu)$ be a charge space and let $f$ be a real-valued function on $X$. Then the following three statements are logically equivalent.
\begin{enumerate}
\item $f \in L_0(X,\calA,\mu)$.
\item $f: X \rightarrow \R$ satisfies the following:
\begin{enumerate}
\item there exists a countable set $C \subset \R$ such that $f^{-1}(y,\infty) \in \overline{\calA}$ for $y \in \R \setminus C$, and
\item $f$ is smooth.
\end{enumerate}
\item There exist sequences of simple functions (with respect to $\overline{\calA}$) $\{ f^+_n \}_{n=1}^{\infty}$ and $\{ f^-_n \}_{n=1}^{\infty}$ such that for each $n \in \N$,
\[
f^+_n := \sum_{j=1}^{K_n-1} y_{n,j} I_{(f^+)^{-1}(y_{n,j},y_{n,j+1}]} + y_n I_{(f^+)^{-1}(y_n,\infty)} 
\] 
and
\[
f^-_n := \sum_{j=1}^{K_n-1} y_{n,j} I_{(f^-)^{-1}(y_{n,j},y_{n,j+1}]} + y_n I_{(f^-)^{-1}(y_n,\infty)}
\] 
where
\begin{enumerate}
\item $K_n := n2^n$, 
\item $y_{n,j} \in ((j-1)2^{-n}, j2^{-n}]$ for each $n \in \N$ and $j \in \{ 1, \ldots, K_n \}$, with $y_n := y_{n,K_n}$,
\item $(f^+)^{-1}(y_{n,j},\infty) \in \overline{\calA}$ and $(f^-)^{-1}(y_{n,j},\infty) \in \overline{\calA}$ for all $n \in \N$ and $j \in \{ 1, \ldots, K_n \}$,
\item $\{ y_{n,j} \}_{j=1}^{K_n} \subset \{ y_{n+1,j} \}_{j=1}^{K_{n+1}}$ for each $n \in \N$, 
\item $f^+_n \xrightarrow{h,\calA} f^+$ and $f^-_n \xrightarrow{h,\calA} f^-$ as $n \rightarrow \infty$.
\end{enumerate} 
\end{enumerate}
\end{theorem}

\begin{proof}
$(1 \implies 2)$ Property~2a is immediate from Lemma~\ref{countableexceptions}, and any $T_1$-measurable function is smooth.

$(2 \implies 3)$ Let $D := \{ |y| : y \in C \}$. For each $n \in \N$ and $j \in \{ 1, \ldots, K_n \}$, inductively choose $y_{n,j}$ to be any element of $((j-1)2^{-n}, j2^{-n}] \setminus D$, unless there is already an element of that set in $\{ y_{n-1,j} \}_{j=1}^{K_{n-1}}$, in which case set $y_{n,j}$ equal to that element. Properties~3a to~3d are immediate consequences of this construction, and then $\{ f^+_n \}_{n=1}^{\infty}$ and $\{ f^-_n \}_{n=1}^{\infty}$ are simple functions with respect to $\overline{\calA}$. 

For any $x \in (f^+)^{-1}[0,y_n]$, $| f^+_n(x) - f^+(x) | < 2^{1-n}$. Also, $f^+$ is smooth since $f^+ \leq f$. Thus for any $\epsilon > 0$ and $n > 1 - \log_2 \epsilon$,
\[
\mu^*(\{ x \in X : | f^+_n(x) - f^+(x) | > \epsilon \}) \leq \mu^*((f^+)^{-1}(y_n,\infty)) \rightarrow 0
\]
as $n \rightarrow \infty$. Hence $f^+_n \xrightarrow{h,\calA} f^+$ as $n \rightarrow \infty$, and similarly $f^-_n \xrightarrow{h,\calA} f^-$ as $n \rightarrow \infty$.

$(3 \implies 1)$ Property~3c implies $\{ f^+_n \}_{n=1}^{\infty}$ is a sequence of simple functions with respect to $\overline{\calA}$, and by 3e and Theorem~\ref{prop_1.8_basile2000}, $f^+_n \xrightarrow{h,\overline{\calA}} f^+$. Hence $f^+ \in L_0(X,\overline{\calA},\overline{\mu}) = L_0(X,\calA,\mu)$, and similarly $f^- \in L_0(X,\calA,\mu)$. By Theorem~\ref{T1_properties}, $f = f^+ - f^- \in L_0(X,\calA,\mu)$. 
\qed \end{proof}

Note that $f^+_n$ and $f^-_n$ can be alternatively defined as
\[
f^+_n := \sum_{j=1}^{K_n-1} y_{n,j} I_{(f^+)^{-1}(y_{n,j},y_{n,j+1}]}
\] 
and
\[
f^-_n := \sum_{j=1}^{K_n-1} y_{n,j} I_{(f^-)^{-1}(y_{n,j},y_{n,j+1}]}
\] 
and the proof goes through unchanged. In fact, it doesn't matter what constant values are assigned to $(f^+)^{-1}(y_n,\infty)$ and $(f^-)^{-1}(y_n,\infty)$. The alternative definitions just provided are used only in the proof of part of Theorem~\ref{integrable_characterisation}, whereas the definitions given in Statement~3 above are used in other proofs.

The above theorem determines the following conditions under which $T_1$-measurability is equivalent to conventional measurability, and integrability with respect to a charge is equivalent to conventional Lebesgue integrability.

\begin{corollary} \label{measurable_fn}
Suppose $(X,\calA,\mu)$ is a complete measure space. Then a function $f : X \rightarrow \R$ is $T_1$-measurable if and only if it is measurable in the conventional sense. 
\end{corollary}

\begin{proof}
$(\implies)$ By Lemma~\ref{equivalent_completeness}, $\overline{\calA} = \calA$. By Theorem~\ref{T1measurable_characterisation}, $f^{-1}(y,\infty) \in \calA$ for any $y \in \R \setminus C$. But then $f^{-1}[y,\infty) \in \calA$ for any $y \in \R$, using the closure of $\calA$ under countable intersections. Hence $f$ is measurable.

$(\impliedby)$ Any measurable function immediately satisfies Statement~2a of Theorem~\ref{T1measurable_characterisation} with $C = \emptyset$. Moreover, any measurable function is smooth by Proposition~4.2.17 of \cite{bhaskararao1983}.
\qed \end{proof}

\begin{corollary} \label{integrable_fn}
Suppose $(X,\calA,\mu)$ is a complete measure space. Then a function $f : X \rightarrow \R$ is integrable (in the sense of Definition~\ref{integrable}) if and only if it is Lebesgue integrable. Moreover, the two types of integral are equal.
\end{corollary}

\begin{proof}
It is sufficient to prove the corollary for non-negative functions, since then it will apply to $f^+$ and $f^-$, and by Theorem~\ref{L1_integrable_properties}, to $f$. Hence assume $f \geq 0$.

Recall that for a complete measure space, $\overline{\calA} = \calA$ by Lemma~\ref{equivalent_completeness}, so simple functions with respect to $\overline{\calA}$ are also simple with respect to $\calA$. Moreover, all simple functions are trivially integrable in both senses, with the two types of integral equal. Also note $f$ is both $T_1$-measurable and measurable in the conventional sense if it is integrable in either sense, by Corollary~\ref{measurable_fn}.

$(\implies)$ By Theorem~\ref{L1_integrable_properties}, any simple function $s$ with $0 \leq s \leq f$ satisfies $0 \leq \int s d\mu \leq \int f d\mu$, where the integral is as in Definition~\ref{integrable}. Hence 
\[
\sup \left\{ \int s d\mu : s \mbox{ simple }, 0 \leq s \leq f \right\} \leq \int f d\mu
\]
is finite, that is, $f$ is Lebesgue integrable. 

$(\impliedby)$ Let $f_n$ be the sequence of simple functions converging monotonically pointwise to $f$ asserted in Statement~3 of Theorem~\ref{T1measurable_characterisation}. By the monotone convergence theorem for the Lebesgue integral, 
\[
\int f_n d\mu \rightarrow \sup \left\{ \int s d\mu : s \mbox{ simple }, 0 \leq s \leq f \right\}.
\]
Hence $\{ \int f_n d\mu \}_{n=1}^{\infty}$ is a Cauchy sequence, and by Definition~\ref{integrable}, $f \in L_1(X, \calA, \mu)$ with
\[
\int f d\mu = \sup \left\{ \int s d\mu : s \mbox{ simple }, 0 \leq s \leq f \right\}.
\]
\qed \end{proof}

These two corollaries together imply that the $L_p$ spaces of Definition~\ref{Lp} are equivalent to the conventional Lebesgue function spaces for a complete measure space, in the sense that they contain the same functions and have the same pseudometric. This in turn gives the following.

\begin{corollary} \label{Riesz-Fischer}
Suppose $(X,\calA,\mu)$ is a complete measure space. Then $L_p(X,\calA,\mu)$ is complete for all $p \in \{0\} \cup [1,\infty)$.
\end{corollary}

\begin{proof}
Conventional Lebesgue function spaces are complete, a result sometimes called the Riesz-Fischer theorem. Hence $L_p(X,\calA,\mu)$ is complete by Corollaries~\ref{integrable_fn} and~\ref{measurable_fn}.
\qed \end{proof}

Versions of Corollaries~\ref{measurable_fn}, \ref{integrable_fn}, and~\ref{Riesz-Fischer} can presumably also be proved for unbounded complete measures. However, the above statements are adequate for the present purposes.

The next and final corollary to Theorem~\ref{T1measurable_characterisation} is used in the proof of Theorem~\ref{null_sets_add_no_functions} to establish that a certain vector space isomorphism is also an isometry.

\begin{corollary} \label{nested_fields_2}
Suppose $(X,\calA,\mu)$ and $(X,\calA^{\prime},\mu)$ are charge spaces with $\calA^{\prime} \subseteq \calA$. Then $d_{\calA^{\prime}}(f,g) = d_{\calA}(f,g)$ for all $f, g \in L_0(X,\calA^{\prime},\mu)$.
\end{corollary}

\begin{proof}
Define $h := | f - g | \in L_0(X,\calA^{\prime},\mu)$. By Theorem~\ref{T1measurable_characterisation}, $h^{-1}(y,\infty) \in \overline{\calA^{\prime}}$ for all $y \in (0,\infty) \setminus C$, where $C \subset (0,\infty)$ is countable. It follows that 
\[
\mu^*_{\calA^{\prime}}(h^{-1}(y,\infty)) = \mu^*_{\calA}(h^{-1}(y,\infty)) = \overline{\mu}(h^{-1}(y,\infty))
\]
for all $y \in (0,\infty) \setminus C$. But then $d_{\calA^{\prime}}(f,g) = d_{\calA}(f,g)$, since $(0,\infty) \setminus C$ is dense in $(0,\infty)$.
\qed \end{proof}

The following theorem asserts that chains of inverse images can also be used to characterise integrable functions, in combination with a regularity condition that is stronger than smoothness (Condition~2b below).

\begin{theorem} \label{integrable_characterisation}
Let $(X,\calA,\mu)$ be a charge space and let $f$ be a real-valued function on $X$. Then the following three statements are logically equivalent.
\begin{enumerate}
\item $f \in L_1(X,\calA,\mu)$.
\item $f: X \rightarrow \R$ satisfies the following:
\begin{enumerate}
\item there exists a countable set $C \subset \R$ such that $f^{-1}(y,\infty) \in \overline{\calA}$ for $y \in \R \setminus C$, and
\item for all $\epsilon > 0$, there exists $y \in (0, \infty)$ such that $| f | I_{| f |^{-1}(y, \infty)} \in L_1(X,\calA,\mu)$ with
\[
\int | f | I_{| f |^{-1}(y, \infty)} d\mu < \epsilon.
\]
\end{enumerate}
\item $f \in L_0(X,\calA,\mu)$ and the sequences $\{ f_n^+ \}_{n=1}^{\infty}$ and $\{ f_n^- \}_{n=1}^{\infty}$ obtained by applying Statement~3 of Theorem~\ref{T1measurable_characterisation} to $f$ satisfy the additional conditions $\int | f^+_m - f^+_n | d\mu \rightarrow 0$ and $\int | f^-_m - f^-_n | d\mu \rightarrow 0$ as $m, n \rightarrow \infty$.
\end{enumerate}
\end{theorem}

\begin{proof}
$(1 \implies 2)$ Any $f \in L_1(X,\calA,\mu)$ is $T_1$-measurable, hence 2a holds by Theorem~\ref{T1measurable_characterisation}. 

For each $n \in \N$, define
\[
f^+_n := \sum_{j=1}^{K_n-1} y_{n,j} I_{(f^+)^{-1}(y_{n,j},y_{n,j+1}]}
\] 
as in the note following Theorem~\ref{T1measurable_characterisation}. By Theorem~\ref{L1_integrable_properties}, $f^+ \in L_1(X,\calA,\mu)$. Thus by dominated convergence (Theorem~\ref{dominated_convergence}), $\int |f^+_n - f^+| d\mu \rightarrow 0$ as $n \rightarrow \infty$. Fix $\epsilon > 0$, then there is $k \in \N$ such that $\int | f^+_k - f^+ | d\mu < \epsilon/2$. Set $y := y_k$ as defined in Theorem~\ref{T1measurable_characterisation}. Then $(f^+)^{-1}(y,\infty) \in \overline{\calA}$ and hence by Corollary~\ref{fI_A}, $f^+ I_{(f^+)^{-1}(y,\infty)} \in L_1(X,\calA,\mu)$. Now
\begin{eqnarray*}
\int f^+ I_{(f^+)^{-1}(y,\infty)} d\mu & = & \int (f^+ - f^+_k) I_{(f^+)^{-1}(y,\infty)} d\mu \\
& = & \int | f^+_k - f^+ | I_{(f^+)^{-1}(y,\infty)} d\mu \\ 
& \leq & \int | f^+_k - f^+ | d\mu \\ 
& < & \epsilon/2
\end{eqnarray*}
where the first equality follows since $f^+_k I_{(f^+)^{-1}(y_k,\infty)} = 0$, and the second follows since $f^+_k \leq f^+$. Similarly, $f^- I_{(f^-)^{-1}(y,\infty)} \in L_1(X,\calA,\mu)$ with 
\[
\int f^- I_{(f^-)^{-1}(y,\infty)} d\mu < \epsilon/2.
\]
Hence $| f | I_{| f |^{-1}(y, \infty)} \in L_1(X,\calA,\mu)$ with
\begin{eqnarray*}
\int | f | I_{| f |^{-1}(y,\infty)} d\mu & = & \int f^+ I_{(f^+)^{-1}(y,\infty)} d\mu + \int f^- I_{(f^-)^{-1}(y,\infty)} d\mu \\
& < & \epsilon.
\end{eqnarray*}

$(2 \implies 3)$ Let $D := \{ |y| : y \in C \}$ and fix $\epsilon > 0$. Then there is $y \in [1, \infty) \setminus D$ such that $| f | I_{| f |^{-1}(y, \infty)} \in L_1(X,\calA,\mu)$ and
\[
\int | f | I_{| f |^{-1}(y, \infty)} d\mu < \epsilon.
\]
Observe that
\begin{eqnarray*}
\int | f | I_{| f |^{-1}(y,\infty)} d\mu & \geq & y \int I_{| f |^{-1}(y,\infty)} d\mu \\
& = & y \int I_{| f |^{-1}(y,\infty)} d\overline{\mu} \\
& = & y \overline{\mu}(| f |^{-1}(y,\infty)) \\
& = & y \mu^*(| f |^{-1}(y,\infty)) \\
& \geq & \mu^*(| f |^{-1}(y,\infty))
\end{eqnarray*}
since $y \geq 1$, using Theorem~\ref{prop_1.8_basile2000} to interchange integration with respect to $\mu$ and $\overline{\mu}$. Hence $\mu^*(| f |^{-1}(y,\infty)) < \epsilon$, that is $f$ is smooth, and $f \in L_0(X,\calA,\mu)$ by Theorem~\ref{T1measurable_characterisation}. 

Let $\{ f^+_n \}_{n=1}^{\infty}$ be the sequence of functions asserted by Statement~3 of Theorem~\ref{T1measurable_characterisation} and note that by Theorem~\ref{T1_properties}, $f^+ \in L_0(X,\calA,\mu)$. Fix $\epsilon > 0$ and choose a new $y \in [0,\infty) \setminus D$ such that $| f | I_{| f |^{-1}(y,\infty)} \in L_1(X,\calA,\mu)$ and
\[
\int | f | I_{| f |^{-1}(y,\infty)} d\mu < \frac{\epsilon}{2}.
\]
By Corollary~\ref{fI_A}, $f^+ I_{(f^+)^{-1}(y,\infty)} \in L_1(X,\calA,\mu)$, and by Theorem~\ref{L1_integrable_properties},
\[
\int f^+ I_{(f^+)^{-1}(y,\infty)} d\mu \leq \int | f | I_{| f |^{-1}(y,\infty)} d\mu < \frac{\epsilon}{2}.
\]
Choose $k \in \N$ such that $y_k > y$ and $2^{1-k} < \epsilon/(2 \mu(X))$. Then for $m> k$ and $n > k$,
\begin{eqnarray*}
\int | f^+_m - f^+_n | d\mu & = & \int | f^+_m - f^+_n | I_{(f^+)^{-1}[0,y])} d\mu + \int | f^+_m - f^+_n | I_{(f^+)^{-1}(y,\infty)} d\mu \\
& \leq & 2^{1-k} \overline{\mu}((f^+)^{-1}[0,y]) + \int f^+ I_{(f^+)^{-1}(y,\infty)} d\mu \\
& < & \epsilon.
\end{eqnarray*}
Thus $\int | f^+_m - f^+_n | d\mu \rightarrow 0$ as $m, n \rightarrow \infty$. Similarly, $\int | f^-_m - f^-_n | d\mu \rightarrow 0$ as $m, n \rightarrow \infty$.

$(3 \implies 1)$ The functions $\{ f^+_n \}_{n=1}^\infty$ are simple functions with respect to $\overline{\calA}$ and hence form a determining sequence for $f^+$ with respect to $\overline{\calA}$. Thus $f^+ \in L_1(X,\overline{\calA},\overline{\mu}) = L_1(X,\calA,\mu)$, by Theorem~\ref{prop_1.8_basile2000}. Similarly $f^- \in L_1(X,\calA,\mu)$. By Theorem~\ref{L1_integrable_properties}, $f = f^+ - f^- \in L_1(X,\calA,\mu)$.
\qed \end{proof}

Again note that $f^+_n$ and $f^-_n$ can be alternatively defined as 
\[
f^+_n := \sum_{j=1}^{K_n-1} y_{n,j} I_{(f^+)^{-1}(y_{n,j},y_{n,j+1}]}
\] 
and
\[
f^-_n := \sum_{j=1}^{K_n-1} y_{n,j} I_{(f^-)^{-1}(y_{n,j},y_{n,j+1}]}
\] 
and the proof goes through unchanged. 

Condition~2 of the above theorem remains sufficient for $f \in L_1(X,\calA,\mu)$ if 2b is weakened to hold only for {\em some} $\epsilon > 0$. For then there exists $y \in (0,\infty)$ such that $| f | I_{| f |^{-1}(y, \infty)} \in L_1(X,\calA,\mu)$ and one may apply the theorem as currently stated to the latter function to obtain that 2b holds for {\em all} $\epsilon > 0$.

The following technical corollary is useful for bounding integrals. It generalises the easy fact that when one charge dominates another, the integrals of the former dominate the integrals of the latter. (This result has been useful to the author in concurrent work, and may be of more general interest.)

\begin{corollary} \label{orderintegrals}
Let $\mu_1$ and $\mu_2$ be charges defined on a common field of subsets $\calA$ of a set $X$. Consider $f : X \rightarrow [0,\infty)$ such that $f \in L_1(X,\calA,\mu_1)$ and $f \in L_1(X,\calA,\mu_2)$. Suppose $E$ is a dense subset of $(0,\infty)$ such that $f^{-1}(y,\infty) \in \overline{\calA}$ for all $y \in E$, and define $\calT := \{ f^{-1}(y,\infty) : y \in E \}$. If $\overline{\mu_1}(A) \leq \overline{\mu_2}(A)$ for each $A \in \calT$, then
\begin{enumerate}
\item $\int f d\mu_1 \leq \int f d\mu_2$, and
\item for each $A \in \calT$, $f I_A \in L_1(X,\calA,\mu_1)$ and $f I_A \in L_1(X,\calA,\mu_2)$ with 
\[
\int f I_A d\mu_1 \leq \int f I_A d\mu_2.
\] 
\end{enumerate}
\end{corollary}

\begin{proof}
Fix $\epsilon > 0$. By Property~2b of Theorem~\ref{integrable_characterisation} and the density of $E$ in $(0,\infty)$, there is $y \in E$ such that $f I_{f^{-1}(y,\infty)} \in L_1(X,\calA,\mu_1)$ with $\int f I_{f^{-1}(y,\infty)} d\mu_1 < \frac{\epsilon}{2}$. Moreover, by the density of $E$ in $(0,\infty)$, there is an increasing sequence $\{ y_j \}_{j=1}^K \subseteq E$ with $K \in \N$, $y_K := y$ and $y_j \in ((j-1)\delta, j\delta]$ for each $j \in \{ 1, \ldots, K \}$, where $\delta := \frac{\epsilon}{4 \mu_1(X)}$. Form the function:
\[
f_{\epsilon} := \sum_{j=1}^{K-1} y_j I_{A_j \setminus A_{j+1}} + y I_{A_K}
\] 
where $A_j := f^{-1}(y_j,\infty) \in \overline{\calA}$ for each $j \in \{ 1, \ldots, K \}$. Note $0 \leq f(x) - f_{\epsilon}(x) < \frac{\epsilon}{2 \mu_1(X)}$ for all $x \in [0,y_K]$.

Then
\begin{eqnarray*}
\int f d\mu_2 & \geq & \int f_{\epsilon} d\mu_2 \\
& = & \sum_{j=1}^{K-1} y_j (\overline{\mu_2}(A_j) - \overline{\mu_2}(A_{j+1})) + y_K \overline{\mu_2}(A_K) \\
& = & y_1 \overline{\mu_2}(A_1) + \sum_{j=2}^{K} (y_j - y_{j-1}) \overline{\mu_2}(A_j) \\
& \geq & y_1 \overline{\mu_1}(A_1) + \sum_{j=2}^{K} (y_j - y_{j-1}) \overline{\mu_1}(A_j) \\
& = & \sum_{j=1}^{K-1} y_j (\overline{\mu_1}(A_j) - \overline{\mu_1}(A_{j+1})) + y_K \overline{\mu_1}(A_K) \\
& = & \int f_{\epsilon} d\mu_1 \\
& \geq & \int f_{\epsilon}I_{A_K^c} d\mu_1 \\
& > & \int \left( f - \frac{\epsilon}{2 \mu_1(X)} \right) I_{A_K^c} d\mu_1 \\  
& = & \int f d\mu_1 - \int f I_{A_K} d\mu_1 - \frac{\epsilon}{2 \mu_1(X)} \overline{\mu_1}(A_K^c) \\
& > & \int f d\mu_1 - \epsilon
\end{eqnarray*}
where the interchange of integration with respect to $\mu$ and $\overline{\mu}$ is justified by Theorem~\ref{prop_1.8_basile2000}. Let $\epsilon \rightarrow 0$ to obtain $\int f d\mu_1 \leq \int f d\mu_2$.

Note that for any $A \in \calT$, the function $f I_A$ satisfies:
\begin{enumerate}
\item $f I_A \in L_1(X,\calA,\mu_1)$ and $f I_A \in L_1(X,\calA,\mu_2)$, 
\item $(f I_A)^{-1}(y,\infty) \in \overline{\calA}$ for all $y \in E$, and
\item $\overline{\mu_1}((f I_A)^{-1}(y,\infty)) \leq \overline{\mu_2}((f I_A)^{-1}(y,\infty))$ for all $y \in E$.
\end{enumerate}
Claim~1 follows by Corollary~\ref{fI_A}. Claims~2 and~3 follow since $(f I_A)^{-1}(y,\infty) = f^{-1}(y,\infty) \in \calT$ if $A \supseteq f^{-1}(y,\infty)$ and $(f I_A)^{-1}(y,\infty) = A \in \calT$ if $A \subseteq f^{-1}(y,\infty)$. Hence by the above reasoning $\int f I_A d\mu_1 \leq \int f I_A d\mu_2$.
\qed \end{proof}

Lemma~\ref{countableatoms} also implies the following characterisation of equality a.e.

\begin{theorem} \label{equalityae_characterisation}
Consider a charge space $(X,\calA,\mu)$ and functions $f : X \rightarrow \R$ and $g : X \rightarrow \R$. Then the following statements hold.
\begin{enumerate}
\item If $f = g$ a.e, then $f I_A = g I_A$ a.e. for any $A \in \calP(X)$, 
\item if $A, B \in \calP(X)$ with $\mu^*(A \triangle B) = 0$, then $f I_A = f I_B$ a.e.,
\item If $f \in L_0(X,\calA,\mu)$, there exists a countable $C \subset \R$ such that the following statements are logically equivalent:
\begin{enumerate}
\item $f = g$ a.e,
\item $\mu^*(f^{-1}(y,\infty) \triangle g^{-1}(y,\infty)) = 0$ for all $y \in \R \setminus C$, and
\item $f I_{f^{-1}(y,\infty)} = g I_{g^{-1}(y,\infty)}$ a.e. for all $y \in \R \setminus C$.
\end{enumerate}
\end{enumerate}
\end{theorem}

\begin{proof}
For 1, note $| f I_A - g I_A | = | f - g | I_A \leq | f - g |$. Hence for any $\epsilon > 0$, 
\[
\mu^*(\{ x : | (f I_A)(x) - (g I_A)(x) | > \epsilon \}) \leq \mu^*(\{ x : | f(x) - g(x) | > \epsilon \}) = 0.
\]

For 2, note 
\[
| f I_A - f I_B | = | f | | I_A - I_B | = | f | I_{A \triangle B}.
\]
Hence for any $\epsilon > 0$,
\[
\mu^*(\{ x : | (f I_A)(x) - (f I_B)(x) | > \epsilon \}) \leq \mu^*(A \triangle B) = 0.
\]

For 3, let $C$ be the countable subset of $\R$ mentioned in Lemma~\ref{countableatoms}. 

$(3a \implies 3b)$ Choose $y \in \R \setminus C$ and $\epsilon > 0$, then there exists $\delta > 0$ such that $\mu^*(f^{-1}[y-\delta,y+\delta]) < \epsilon$. Now
\begin{eqnarray*}
x \in f^{-1}(y,\infty) \triangle g^{-1}(y,\infty) & \iff & (f(x) > y \mbox{ and } g(x) \leq y ) \mbox{ or } \\
& & (f(x) \leq y \mbox{ and } g(x) > y ) \\
& \implies & x \in f^{-1}[y-\delta,y+\delta] \mbox{ or } \\
& & | f(x) - g(x) | > \delta,
\end{eqnarray*}
hence 
\[
\mu^*(f^{-1}(y,\infty) \triangle g^{-1}(y,\infty)) \leq \mu^*(f^{-1}[y-\delta,y+\delta]) < \epsilon.
\]
Let $\epsilon \rightarrow 0$ to obtain the result.

$(3b \implies 3a)$ For each $x \in X$ with $g(x) \neq f(x)$, one of the following cases applies. Consider the case $g(x) > f(x)$. Then for any $y \in (f(x), g(x)) \setminus C$, $x \in g^{-1}(y,\infty) \setminus f^{-1}(y,\infty)$. Next consider the case $g(x) < f(x)$. Then for any $y \in (g(x), f(x)) \setminus C$, $x \in f^{-1}(y,\infty) \setminus g^{-1}(y,\infty)$. In either case, $x \in f^{-1}(y,\infty) \triangle g^{-1}(y,\infty)$. Fix $\epsilon > 0$ and $\delta > 0$. By the smoothness of $f$, there is $y_0 \in (0,\infty)$ such that $\mu^*(\{x : |f(x)| > y_0\}) < \delta$. Moreover, by the density of $C$ in $\R$, there is a finite set $\{ y_j \}_{j=1}^{K} \subset \R \setminus C$ such that $\max \{ | y - y_k | : k \in \{ 1, \ldots, K \} \} < \epsilon$ for all $y \in [-y_0,y_0]$. Then
\begin{eqnarray*}
\mu^*(\{ x : | f(x) - g(x) | > \epsilon\}) & \leq & \mu^*(\{ x \in f^{-1}[-y_0,y_0] : | f(x) - g(x) | > \epsilon\}) \\
& & + \mu^*(\{x : |f(x)| > y_0\}) \\
& < & \sum_{j=1}^K \mu^*(f^{-1}(y_j,\infty) \triangle g^{-1}(y_j,\infty)) + \delta \\
& = & \delta,
\end{eqnarray*}
Letting $\delta \rightarrow 0$ gives $f = g$ almost everywhere. 

$(3a \implies 3c)$ Observe that $f I_{f^{-1}(y,\infty)} = g I_{f^{-1}(y,\infty)}$ a.e. by 1, and $g I_{f^{-1}(y,\infty)} = g I_{g^{-1}(y,\infty)}$ a.e. by 2 and 3b (which is implied by 3a), for any $y \in \R \setminus C$.

$(3c \implies 3b)$ Consider $y \in \R \setminus C$. Then there is $z \in (-\infty,y) \setminus C$. Define $f_z := f I_{f^{-1}(z,\infty)}$ and $g_z := g I_{g^{-1}(z,\infty)}$. Then $f_z = g_z$ a.e. and moreover 
\[
\lim_{\delta \rightarrow 0} \mu^*(f_z^{-1}[y-\delta,y+\delta]) = \lim_{\delta \rightarrow 0} \mu^*(f^{-1}[y-\delta,y+\delta]) = 0.
\]
Reasoning as in $(3a \implies 3b)$ above gives 
\[
\mu^*(f^{-1}(y,\infty) \triangle g^{-1}(y,\infty)) = \mu^*(f_z^{-1}(y,\infty) \triangle g_z^{-1}(y,\infty)) = 0.
\]
\qed \end{proof}

This section has highlighted several properties of the chains of inverse images induced by functions in $L_p$ spaces over a general bounded charge space. Further study of such chains in a general context seems a promising direction for future research, but is beyond the scope of this paper.

\section{Complete $L_p$ spaces}
\label{completeness_section}

It is not in general true that the metric space $L_p(X,\calA,\mu)$ is complete, where $(X,\calA,\mu)$ is a charge space and $p \in \{0\} \cup [1,\infty)$. See~\cite{gangopadhyay1990}, \cite{gangopadhyay1999} and especially~\cite{basile2000} for helpful discussions of complete $L_p$ spaces. However, the wider literature on completeness of $L_p$ spaces should be approached with caution, as it contains a number of errors. For example, Proposition 5.2 in a paper by Armstrong~\cite{armstrong1991} claims the following are logically equivalent for a bounded charge space: (i) $L_1(X,\calA,\mu)$ is complete, (ii) $\mu$ induces a countably additive set function on the Boolean quotient $\calA / \calM$, where $\calM := \mu^{-1}(0)$ is the {\em kernel} of $\mu$, and (iii) $\overline{\calA}$ is a $\sigma$-field. However, as Basile and Bhaskara Rao~\cite{basile2000} show, the only true implication in this proposition is (i) $\implies$ (ii); counter-examples are available for each of the others. This error has been propagated across multiple papers. Another characterisation of completeness, due to Green~\cite{green1971}, was shown to be false in~\cite{bhaskararao1987}.

This section contains a new characterisation of completeness for $L_p$ spaces constructed on a bounded, non-negative charge space (Theorem~\ref{completeness_characterisation}), in terms of a representation of $\calA / \calM$ on a field of sets. In addition, a sufficient condition for completeness is proved, correcting the false condition (ii) $\implies$ (i) of Armstrong just mentioned. The correction involves the auxiliary condition that $\calA / \calM$ be a countably complete Boolean algebra. 

Since this section makes frequent reference to Boolean algebras and their representations, the following brief discussion of relevant terminology and notation, with particular reference to charge spaces, may be helpful. 

Any field of sets $\calA \subseteq \calP(X)$ on an arbitrary set $X$ forms a Boolean algebra with pairwise intersection $\cap$ as the {\em meet} operator $\wedge$, pairwise union $\cup$ as the {\em join} operator $\vee$, set complement $^c$ as the Boolean complement operator~$^{\prime}$, the empty set $\emptyset$ as the zero and $X$ as the unit. Note also that the symmetric set difference $p \triangle q$ is an exclusive disjunction $p + q := (p \wedge q^{\prime}) \vee (p^{\prime} \wedge q)$ and the subset relation $p \subseteq q$ is a Boolean partial order $p \leq q$.

Stone's representation theorem (originally proved in~\cite{stone1937}, but~\cite{givant2009} and~\cite{sikorski1969boolean} contain helpful expositions) is a fundamental result asserting that every abstract Boolean algebra $\calA$ is isomorphic to a field of sets. More specifically, $\calA$ can be embedded in $\calP(S_{\calA})$, where $S_{\calA}$ is a set called the {\em Stone space} of $\calA$. Stone spaces can be constructed in various equivalent ways, but the details are not needed here. It will be sufficient throughout what follows to work with an unspecified {\em representation} of $\calA$, that is, a Boolean isomorphism $\phi : \calA \rightarrow \calP(Y)$ embedding $\calA$ into the power set of some set $Y$.

For any charge space $(X,\calA,\mu)$, the set $\calM := \mu^{-1}(0) := \{ A \in \calA : \mu(A) = 0 \}$, called the {\em kernel} of $\mu$, is a Boolean ideal of $\calA$, and the collection of {\em null sets} $\calN := \{ A \in \calP(X) : \mu^*(A) = 0 \}$ is a Boolean ideal of $\calP(X)$.

Suppose $\calA$ is a Boolean algebra and $\calM$ is a Boolean ideal of $\calA$. Let $[A] \in \calA / \calM$ denote the equivalence class of $A \in \calA$ under the equivalence relation $A \sim B \iff A + B \in \calM$. If the quotient space $\calA / \calM$ is potentially ambiguous, the notation $[A]_{\calA / \calM}$ will be used to clarify.

Suppose $\phi : \calA / \calM \rightarrow \calP(Y)$ is an embedding of the Boolean algebra $\calA / \calM$ in the power set of some set $Y$ (which may or may not be the Stone space $S_{\calA / \calM}$). Define $\langle A \rangle := \phi([A])$ for all $A \in \calA$. If the embedding $\phi$ is potentially ambiguous, it will be indicated by a subscript as $\langle A \rangle_{\phi}$. 

For any $\calA^{\prime} \subseteq \calA$ (not necessarily a sub-algebra), define $[\calA^{\prime}] := \{ [A] : A \in \calA^{\prime} \}$. Similarly, define $\langle \calA^{\prime} \rangle := \{ \langle A \rangle : A \in \calA^{\prime} \}$. 

If $(X,\calA,\mu)$ is a charge space and $\calM$ is the kernel of $\mu$, the induced function $\mu : \calA / \calM \rightarrow \R$ given by $\mu[ A ] := \mu(A)$ for all $A \in \calA$ is finitely additive. Similarly, if $\phi : \calA / \calM \rightarrow \calP(Y)$ is a representation of $\calA / \calM$, one may define a charge $\mu$ on $\langle \calA \rangle_{\phi}$ by $\mu\langle A \rangle_{\phi} := \mu(A)$ for all $A \in \calA$, so that $(Y,\langle \calA \rangle_{\phi},\mu)$ is a charge space. Note that wherever possible in what follows, parentheses enclosing a function argument will be omitted when that argument is already enclosed in square or angled brackets; for example, $\mu[ A ]$ and $\mu\langle A \rangle$ will be preferred to $\mu([ A ])$ and $\mu(\langle A \rangle)$

A Boolean algebra $\calA$ is said to be {\it countably complete} if every countable subset $\{ p_k \}_{k=1}^{\infty}$ in $\calA$ has a least upper bound in $\calA$.

A {\em countably additive function} $\mu$ on a countably complete Boolean algebra $\calA$ is a generalisation of a non-negative measure: it is a function $\mu: \calA \rightarrow \R^+$ such that $\mu(0) = 0$ and $\mu(\vee_{k=1}^{\infty} p_k) = \sum_{k=1}^{\infty} \mu(p_k)$ for any pairwise disjoint sequence $\{ p_k \}_{k=1}^{\infty} \subseteq \calA$. In this paper, it will be assumed that a countably additive function is bounded, that is, $\mu(1) < \infty$, unless otherwise stated.

First it will be useful to show that a representation of a countably complete Boolean algebra can be transformed into a measure space as follows.  

\begin{lemma} \label{sigma_embedding}
Let $\calA$ be a countably complete Boolean algebra and let $\mu : \calA \rightarrow \R^+$ be a countably additive function. Suppose $\phi : \calA \rightarrow \calP(Y)$ is a representation of $\calA$. Then $(Y,\overline{\phi(\calA)},\overline{\mu})$ is a complete measure space, where $\mu : \phi(\calA) \rightarrow \R^+$ is the charge defined by $\mu(\phi(A)) := \mu(A)$ for all $A \in \calA$. 
\end{lemma}

\begin{proof}
Let $\{ A_k \}_{k=1}^{\infty} \subseteq \overline{\phi(\calA)}$ be pairwise disjoint, and set $A := \cup_{k=1}^{\infty} A_k$. For each $k \in \N$ and $\epsilon > 0$, there exist $B_k, C_k \in \calA$ such that $\phi(B_k) \subseteq A_k \subseteq \phi(C_k)$ and $\mu(C_k) - \mu(B_k) < (\epsilon / 2) 2^{-k}$. Then $\{ B_k \}_{k=1}^{\infty}$ are pairwise disjoint, and $\mu(\vee_{k=1}^{\infty} B_k) = \sum_{k=1}^{\infty} \mu(B_k)$. Choose $K \in \N$ such that $\sum_{k=1}^K \mu(B_k) > \sum_{k=1}^{\infty} \mu(B_k) - \epsilon / 2$, and set $B := \vee_{k=1}^{K} B_k$ and $C := \vee_{k=1}^{\infty} C_k$. Then $\phi(B) \subseteq A \subseteq \phi(C)$, where the latter containment follows because $\phi(C)$ is an upper bound for $\{ A_k \}_{k=1}^{\infty}$, but $A$ is the least upper bound for $\{ A_k \}_{k=1}^{\infty}$ in $\calP(Y)$. Moreover,
\[
\mu(C) - \mu(B) \leq \sum_{k=1}^{\infty} (\mu(C_k) - \mu(B_k)) + \left(\sum_{k=1}^{\infty} \mu(B_k) - \mu(B)\right) < \epsilon.
\]
Hence $A \in \overline{\phi(\calA)}$, and $\overline{\phi(\calA)}$ is a monotone class and therefore also a $\sigma$-field (see \cite[Thm. A, p. 27]{halmos2013measure}). Also note 
\[
\overline{\mu}(A) \leq \mu(C) \leq \sum_{k=1}^{\infty} \mu(C_k) < \sum_{k=1}^{\infty} \overline{\mu}(A_k) + \epsilon/2
\]
and 
\[
\sum_{k=1}^{\infty} \overline{\mu}(A_k) = \lim_{k \rightarrow \infty} \overline{\mu}(\cup_{j=1}^k A_k) \leq \overline{\mu}(A),
\]
implying $\overline{\mu}$ is countably additive on $\overline{\phi(\calA)}$. Thus $(Y,\overline{\phi(\calA)},\overline{\mu})$ is a measure space. It is a complete measure space by Lemma~\ref{equivalent_completeness}.
\qed \end{proof}

The preceding lemma provides a sufficient condition for an $L_p$ space to be complete, described in the following theorem, which also identifies a condition that is both necessary and sufficient.

\begin{theorem} \label{completeness_characterisation}
Consider a charge space $(X,\calA,\mu)$, and let $\calM := \mu^{-1}(0)$. Suppose $\phi: \calA /\calM \rightarrow \calP(Y)$ is a representation of $\calA /\calM$. Then the following statements hold.
\begin{enumerate}
\item $L_1(X,\calA,\mu)$ is complete if and only if $(Y,\overline{\langle \calA \rangle_{\phi}}, \overline{\mu})$ is a measure space, where $\mu : \langle \calA \rangle_{\phi} \rightarrow \R^+$ is the charge defined by $\mu\langle A \rangle_{\phi} := \mu(A)$ for all $A \in \calA$. 

\item If $\calA / \calM$ is a countably complete Boolean algebra and the function $\mu : \calA / \calM \rightarrow \R^+$, defined by $\mu[ A ]_{\calA / \calM} := \mu(A)$ for all $A \in \calA$, is a countably additive function, then $L_1(X,\calA,\mu)$ is complete.
\end{enumerate}
\end{theorem}

\begin{proof}
First consider the claimed necessary and sufficient condition.

$(\implies)$ Since $\overline{\langle \calA \rangle}$ is a field, it is a $\sigma$-field if it is closed under countable disjoint unions \cite[Thm. A, p. 27]{halmos2013measure}. Consider pairwise disjoint $\{ A_k \}_{k=1}^{\infty} \subseteq \overline{\langle \calA \rangle}$, and fix $\epsilon > 0$. For each $k$, there exist $B_k, C_k \in \calA$ such that $\langle B_k \rangle \subseteq A_k \subseteq \langle C_k \rangle$ and $\mu(C_k) - \mu(B_k) < 2^{-k} \epsilon/3$. Note that the increasing sequence $\{ \cup_{j=1}^k B_j \}_{k=1}^{\infty}$ is bounded above by $\cup_{j=1}^{\infty} B_j$, hence one may choose $K$ so that $\mu(\cup_{j=1}^K B_j) > \lim_{k \rightarrow \infty} \mu(\cup_{j=1}^k B_j) - \epsilon/3$, and set $B := \cup_{j=1}^K B_j \in \calA$. By Theorem~\ref{Lp_complete}, there exists $C \in \calA$ such that $\mu(\cup_{j=1}^k C_j \setminus C) = 0$ for each $k$ and $\lim_{k \rightarrow \infty} \mu(\cup_{j=1}^k C_j) \leq \mu(C) \leq \lim_{k \rightarrow \infty} \mu(\cup_{j=1}^k C_k) + \epsilon/3$. It follows that $C_k \setminus C \in \calM$ and thus $C_k = (C_k \cap C) \cup (C_k \setminus C) \in [ C_k \cap C ]$, giving $\langle C_k \rangle = \langle C_k \cap C \rangle \subseteq \langle C \rangle$ for all $k$. Putting all this together gives $\langle B \rangle \subseteq \cup_{k=1}^{\infty} A_k \subseteq \langle C \rangle$ and 
\begin{eqnarray*}
\mu(C) - \mu (B) &=& \mu(C) -  \lim_{k \rightarrow \infty} \mu(\cup_{j=1}^k C_j) \\
& & + \lim_{k \rightarrow \infty} \left( \mu(\cup_{j=1}^k C_j) - \mu(\cup_{j=1}^k B_j) \right)\\
& & + \lim_{k \rightarrow \infty} \mu(\cup_{j=1}^k B_j) - \mu(B) \\
& < & \epsilon.
\end{eqnarray*}
Hence $\cup_{k=1}^{\infty} A_k \in \overline{\langle \calA \rangle}$. Moreover, $\lim_{k \rightarrow \infty} \mu(\cup_{j=1}^k C_k) = \sum_{k=1}^{\infty} \mu(C_k)$, giving
\begin{eqnarray*}
\overline{\mu}(\cup_{k=1}^{\infty} A_k) &\leq& \mu\langle C \rangle \\
&\leq& \sum_{k=1}^{\infty} \mu\langle C_k \rangle + \epsilon/3 \\
&\leq& \sum_{k=1}^{\infty} \bigg( \overline{\mu}(A_k) + 2^{-k}\epsilon/3 \bigg) + \epsilon/3 \\
&=& \sum_{k=1}^{\infty} \overline{\mu}(A_k) + 2\epsilon/3.
\end{eqnarray*}
Letting $\epsilon \rightarrow 0$ gives that $\overline{\mu}$ is countably additive on $\overline{\langle \calA \rangle}$.

$(\impliedby)$ Consider an increasing sequence $A_1 \subseteq A_2 \subseteq \ldots$ in $\calA$. Since  $\overline{\langle \calA \rangle}$ is a $\sigma$-field, $\cup_{k=1}^{\infty} \langle A_k \rangle \in \overline{\langle \calA \rangle}$, implying that for any $\epsilon > 0$ there exist $B, C \in \calA$ with $\langle B \rangle \subseteq \cup_{k=1}^{\infty} \langle A_k \rangle \subseteq \langle C \rangle$ and $\mu(C) - \mu(B) < \epsilon$. For each $k$, $\mu(A_k \setminus C) = \mu(\langle A_k \rangle \setminus \langle C \rangle) = 0$ since $\langle A_k \rangle \subseteq \langle C \rangle$. Since $\overline{\mu}$ is countably additive on $\overline{\langle \calA \rangle}$, $\overline{\mu}(\cup_{k=1}^{\infty} \langle A_k \rangle) = \lim_{k \rightarrow \infty} \mu(A_k)$. Combining this with 
\[
\overline{\mu}(\cup_{k=1}^{\infty} \langle A_k \rangle) \leq \mu\langle C \rangle < \overline{\mu}(\cup_{k=1}^{\infty} \langle A_k \rangle) + \epsilon
\]
gives the conditions of Theorem~\ref{Lp_complete}, hence $L_p(X,\calA,\mu)$ is complete.

For the second condition, note that if $\calA / \calM$ is countably complete and $\mu$ is countably additive on $\calA / \calM$, then $(Y,\overline{\langle \calA \rangle},\overline{\mu})$ is a measure space by Lemma~\ref{sigma_embedding}, implying $L_1(X,\calA,\mu)$ is complete.
\qed \end{proof}

Thus one strategy for constructing a complete $L_p(X,\calA,\mu)$ space is to ensure $\calA / \calM$ is countably complete and the induced function $\mu$ is countably additive on $\calA / \calM$. 

\section{$L_p$ spaces over complete charge spaces} 
\label{complete_fields}

Complete measure spaces play an important role in measure theory. This section generalises the concept for charge spaces, and characterises those charge spaces $(X,\calA,\mu)$ for which completion leaves the corresponding $\calL_p(X,\calA,\mu)$ spaces unchanged.

\begin{definition} \label{charge_completeness}
A charge space $(X,\calA,\mu)$ is {\em complete} if it contains its null sets, that is, if $\calN := \{ A \in \calP(X) : \mu^*(A) = 0 \} \subset \calA$.
\end{definition}

If $(X,\calA,\mu)$ is a measure space, then the above definition is equivalent to the standard definition of a complete measure space, since in that case if $A \in \calP(X)$ with $\mu^*(A) = 0$, then $A \subseteq B$ for some $B \in \calA$ with $\mu(B) = 0$. Moreover, for measure spaces, completeness is equivalent to Peano-Jordan completeness by Lemma~\ref{equivalent_completeness}. However, in general, completeness in the above sense is a weaker condition than Peano-Jordan completeness: it is straightforward to show $(X,\calA,\mu)$ is complete if it is Peano-Jordan complete, but the converse does not hold in general. A counter-example to the converse is discussed below, after the proof of Lemma~\ref{completefield}.

The following notation and lemma will be useful in this section.

\begin{definition} \label{generated_field}
For any algebra $\calA$ and any $\calC \subseteq \calA$, let $\alpha(\calC)$ denote the subalgebra of $\calA$ generated by $\calC$, that is, the smallest subalgebra of $\calA$ containing $\calC$.
\end{definition}

\begin{lemma} \label{fieldplusnull}
Suppose $\calA$ and $\calB$ are Boolean algebras with $\calB \subseteq \calA$, and $\calM$ is a Boolean ideal of $\calA$. Then 
\begin{eqnarray*}
\alpha(\calB \cup \calM) &=& \{ (B \wedge M^{\prime}) \vee N : B \in \calB, M,N \in \calM \} \\
&=& \{ A \in \calA : A + B \in \calM \mbox{ for some } B \in \calB \}.
\end{eqnarray*}
\end{lemma}

\begin{proof}
The collection $\calC := \{ (B \wedge M^{\prime}) \vee N : B \in \calA^{\prime}, M,N \in \calM \}$ is a subalgebra of $\calA$. To see this, note:
\begin{enumerate}
\item $X = (X \wedge 0^{\prime}) \vee 0 \in \calC$,
\item $((B_1 \wedge M^{\prime}) \vee N)^{\prime} = (B_1^{\prime} \wedge N^{\prime}) \vee (M \wedge N^{\prime}) \in \calC$, and
\item $((B_1 \wedge M_1^{\prime}) \vee N_1) \vee ((B_2 \wedge M_2^{\prime}) \vee N_2) = (B_3 \wedge M_3^{\prime}) \vee N_3 \in \calC$ where $B_3 := B_1 \vee B_2$, 
$M_3 := (M_1 \wedge M_2) \vee (M_1 \wedge B_2^{\prime}) \vee (M_2 \wedge B_1^{\prime})$
and $N_3 := N_1 \vee N_2$,
\end{enumerate}
for any $B_1,B_2 \in \calB$ and $M_1, N_1, M_2, N_2 \in \calM$. Thus $\calC$ is a Boolean algebra containing $\calB \cup \calM$, hence $\alpha(\calB \cup \calM) \subseteq \calC$. But also $\calC \subseteq \alpha(\calB \cup \calM)$, hence $\calC = \alpha(\calB \cup \calM)$. 

Consider $A \in \alpha(\calB \cup \calM)$. Then there is $B \in \calB$ and $M,N \in \calM$ such that $A := (B \wedge M^{\prime}) \vee N$. But then $A + B \leq M \vee N \in \calM$, implying $A + B \in \calM$. On the other hand, if $A \in \calA$ such that $A + B \in \calM$ for some $B \in \calB$, then $A = (B \wedge M^{\prime}) \vee N$ where $M := B \wedge A^{\prime} \in \calM$ and $N := A \wedge B^{\prime} \in \calM$.
\qed \end{proof}

Any charge space can be completed by adding null sets, in the following manner.

\begin{lemma} \label{completefield}
Let $(X,\calA,\mu)$ be a charge space with null sets $\calN \subset \calP(X)$. Then $\mu$ has a unique extension to a charge on $\alpha(\calA \cup \calN)$. Moreover, $(X,\alpha(\calA \cup \calN),\mu)$ is a complete charge space with null sets $\calN$.
\end{lemma}

\begin{proof}
Since $\calN$ is a Boolean ideal of $\calP(X)$, Lemma~\ref{fieldplusnull} gives 
\[
\alpha(\calA \cup \calN) = \{ A \in \calP(X) : A \triangle B \in \calN \mbox{ for some } B \in \calA \}.
\]
Thus for any $A \in \alpha(\calA \cup \calN)$, there is $B \in \calA$ such that $A \triangle B \in \calN$. One may define $\mu(A) := \mu(B)$. This extension of $\mu$ is well defined, since if there is another element $C \in \calA$ such that $A \triangle C \in \calN$, then $B \triangle C \in \calN$, and $\mu(B) = \mu(C)$. It is then straightforward to show $(X,\alpha(\calA \cup \calN),\mu)$ is a charge space. 

Suppose there is another charge $\mu^{\prime}$ on $\alpha(\calA \cup \calN)$ that agrees with $\mu$ on $\calA$. Then for $A \in \alpha(\calA \cup \calN)$, there is $B \in \calA$ with $A \triangle B \in \calN$. Hence $\mu^{\prime}(A) = \mu(B) = \mu(A)$, implying the extension is unique.

To see that $(X,\alpha(\calA \cup \calN),\mu)$ is complete, consider any null set $N$ of this charge space and fix $\epsilon > 0$. Then there exists $A \in \alpha(\calA \cup \calN)$ with $N \subseteq A$ and $\mu(A) < \epsilon/2$. Moreover, there exists $B \in \calA$ with $A \triangle B \in \calN$. But then there exists $C \in \calA$ with $A \triangle B \subseteq C$ and $\mu(C) < \epsilon/2$. It follows that $N \subseteq B \cup C \in \calA$ with $\mu(B \cup C) \leq \mu(B) + \mu(C) < \epsilon$. Letting $\epsilon \rightarrow 0$ gives $N \in \calN \subset \alpha(\calA \cup \calN)$, implying $\alpha(\calA \cup \calN)$ is complete, with null sets $\calN$.
\qed \end{proof}

The following example demonstrates that a charge space can be complete and yet not Peano-Jordan complete. Consider the charge space $(\N,\calA,\nu)$, where $\N$ is the natural numbers excluding 0 and $\calA$ is the collection of periodic sets in $\N$. Here a set $A \in \calP(\N)$ is said to be {\em periodic} if the binary sequence $(I_A(1),I_A(2),\ldots)$ is periodic, where $I_A$ is the indicator function for set $A$. The charge $\nu$ is defined by
\[
\nu(A) := \lim_{N \rightarrow \infty} \frac{1}{N} \sum_{n=1}^N I_A(n)
\]
for all $A \in \calA$. One can show (proof omitted) that $\nu(\calA)$ is the set of rational numbers in $[0,1]$, whereas $\nu(\overline{\calA}) = [0,1]$, so $\calA$ is not Peano-Jordan complete. Nor is $\calA$ complete, since the set $\calN$ of null sets includes the finite sequences, which are not periodic. However, the charge space $(\N,\alpha(\calA \cup \calN),\nu)$ described in Lemma~\ref{completefield} contains its null sets and is therefore complete, whereas Lemma~\ref{fieldplusnull} implies $\nu(\alpha(\calA \cup \calN))$ is still the set of rational numbers in $[0,1]$, implying $(\N,\alpha(\calA \cup \calN),\nu)$ is not Peano-Jordan complete.

The rest of this section characterises those charge spaces $(X,\calA,\mu)$ for which completion does not add any new equivalence classes to $\calL_p(X,\calA,\mu)$ spaces, though it may add functions to those equivalence classes. To begin, Lemma~\ref{nullmodification_X} describes a construction that is here called a {\em null modification}, mapping chains (totally ordered sets) from a charge space into a subspace, by adding and removing null sets. The proof of Lemma~\ref{nullmodification_X} relies on the following real analysis lemma.

\begin{lemma} \label{countable_subset} 
Totally ordered sets have the following properties.
\begin{enumerate}
\item Any $\calT \subseteq \R$ contains a countable subset $\calC$ such that for $a \in \calT$ and $\epsilon > 0$, there exist $b, c \in \calC$ with $b \leq a \leq c$ and $c - b < \epsilon$. 
\item Consider a totally ordered set $\calT$ and a strictly increasing function $\mu : \calT \rightarrow \R$. Then $\calT$ contains a countable subset $\calC$ such that for $A \in \calT$ and $\epsilon > 0$, there exist $B, C \in \calC$ with $B \leq A \leq C$ and $\mu(C) - \mu(B) < \epsilon$.
\end{enumerate}
\end{lemma}

\begin{proof}
For 1, let $\calC_0$ be a countable, dense subset of $\calT$. (Such a subset exists since any subset of the reals is separable. Standard proofs of this invoke the axiom of countable choice.) Define 
\begin{align*}
\calC_1 &:= \{ c \in \calC : \exists \epsilon > 0 \mbox{ with } (c - \epsilon/2,c) \cap \calC_0 = \emptyset \}, \mbox{ and } \\ 
\calC_2 &:= \{ c \in \calC : \exists \epsilon > 0 \mbox{ with } (c, c + \epsilon/2) \cap \calC_0 = \emptyset \}.
\end{align*}
Then $\calC_1$ and $\calC_2$ are both countable, since there cannot be an uncountable number of pairwise disjoint intervals of non-zero width contained in $\R$ (each must contain a distinct rational). Thus $\calC := \calC_0 \cup \calC_1 \cup \calC_2$ is a countable set with the claimed property.

Claim~2 then follows by applying Claim~1 to $\mu(\calT)$, and noting $\mu$ is one-to-one and order preserving, as a consequence of being strictly increasing.
\qed \end{proof}

It will be convenient to introduce the following notation. Define
\begin{eqnarray*}
\calC(A,B) & := & \{ C \in \calC : A \subset C \subset B \} 
\end{eqnarray*}
to represent {\em sub-intervals} of a chain $\calC \subset \calA$. Here $A, B \in \calA$ but are not necessarily elements of $\calC$. 

\begin{lemma} \label{nullmodification_X}
Suppose $(X,\calA,\mu)$ is a Peano-Jordan complete charge space with null sets $\calN$, and $\calA^{\prime}$ is a sub-field of $\calA$ such that $\overline{\alpha(\calA^{\prime} \cup \calN)} = \alpha(\overline{\calA^{\prime}} \cup \calN)$. Given a chain $\calT \subseteq \overline{\alpha(\calA^{\prime} \cup \calN)}$, there exists a map $\phi : \calT \rightarrow \overline{\calA^{\prime}}$ such that:
\begin{enumerate}
\item for all $A \in \calT$, $\phi(A) \triangle A \in \calN$,
\item for all $A, B \in \calT$, 
\begin{align*}
[A]_{\calA / \calN} = [B]_{\calA / \calN} &\iff \phi(A) = \phi(B) \mbox{ and } \\
[A]_{\calA / \calN} < [B]_{\calA / \calN} &\iff \phi(A) \subset \phi(B).
\end{align*}
\end{enumerate}
\end{lemma}

\begin{proof}
%
First note that since $(X,\calA,\mu)$ is Peano-Jordan complete, it contains its null sets and hence $\calT \subseteq \overline{\alpha(\calA^{\prime} \cup \calN)} \subseteq \calA$.

Let $[\calC]$ be the countable subchain of $[\calT] := \{ [A] : A \in \calT \}$ obtained by applying Lemma~\ref{countable_subset}(2) to the totally ordered set $[\calT]$ and the strictly increasing function $\nu: [\calT] \rightarrow \R$ given by $\nu[A] := \nu(A)$. Then let $\calC \subseteq \calT$ be obtained by selecting exactly one element of $\calT$ from each of the equivalence classes in $[\calC]$. (This implicitly invokes the axiom of countable choice, in general.)

Arbitrarily order the elements of $\calC$ as a sequence $\{ C_k \}_{k=1}^{\infty}$. (If $\calC$ is finite, the proof below still holds with minimal modifications.) For each $k \in \N$, define $B_k$ to be the largest set in the sequence $C_1, \ldots, C_{k-1}$ that is a proper subset of $C_k$, if such a set exists (that is, $B_k := \bigcup \{ C_j : j < k, C_j \subset C_k  \}$). Otherwise, set $B_k := \emptyset$. Similarly, define $D_k$ to be the smallest set in the sequence $C_1, \ldots, C_{k-1}$ that properly contains $C_k$, if such a set exists (that is, $D_k := \bigcap \{ C_j : j < k, C_j \supset C_k  \}$). Otherwise, set $D_k := X$.

Set $\phi_0(C) = C$ for all $C \in \calC$, and suppose inductively that $\phi_{k-1}(\calC) \subseteq \overline{\alpha(\calA^{\prime} \cup \calN)}$. This is trivially true for $k = 1$, since $\calT \subseteq \overline{\alpha(\calA^{\prime} \cup \calN)}$. Sequentially define $\phi_k$ for each $k \in \N$ as follows:
\begin{enumerate}
\item Select $E_k \in \overline{\calA^{\prime}}$ such that $E_k \triangle \phi_{k-1}(C_k) \in \calN$.
\item Form the set $F_k := \big( E_k \cup \phi_{k-1}(B_k) \big) \cap \phi_{k-1}(D_k)$.
\item For each $C \in \calC$, define
\[
\phi_k(C) := \left\{
\begin{array}{ll}
F_k & \mbox{ if } C = C_k \\
\phi_{k-1}(C) \cap F_k & \mbox{ if } C \in \calC(B_k, C_k) \\
\phi_{k-1}(C) \cup F_k & \mbox{ if } C \in \calC(C_k, D_k) \\
\phi_{k-1}(C) & \mbox{ otherwise.}
\end{array}
\right.
\]
\end{enumerate}
The set $E_k \in \overline{\calA^{\prime}}$ selected in Step~1 exists because, by assumption,  $\overline{\alpha(\calA^{\prime} \cup \calN)} = \alpha(\overline{\calA^{\prime}} \cup \calN)$, hence Lemma~\ref{fieldplusnull} applies. Notice that for each $C \in \calC$, $\phi_k(C)$ is formed from sets of $\overline{\alpha(\calA^{\prime} \cup \calN)}$ using finitely many basic set operations, hence $\phi_k(\calC) \subseteq \overline{\alpha(\calA^{\prime} \cup \calN)}$.

It is straightforward to show (by induction) that $\phi_k$ is an order preserving map for each $k \in \N$, in the sense $B \subseteq C \implies \phi_k(B) \subseteq \phi_k(C)$ for all $B, C \in \calC$, and that $\phi_j(C_k) = \phi_k(C_k)$ for all $k \in \N$ and $j \geq k$. 

Define $\phi^{\prime}(C_k) := \phi_k(C_k)$ for all $k \in \N$. Then $\phi^{\prime}$ is order preserving, since if $C_j \subseteq C_k$,  
\[
\phi^{\prime}(C_j) = \phi_{\max\{j,k\}} (C_j) \subseteq \phi_{\max\{j,k\}} (C_k) = \phi^{\prime}(C_k).
\]
Thus $\phi^{\prime}(\calC)$ is a chain.

Assume $\phi_{k-1}(C_j) \in \overline{\calA^{\prime}}$ for $j \in \{ 1, \ldots, k-1 \}$, noting this is trivially satisfied when $k = 1$ (because $\{ 1, \ldots, k-1 \} = \emptyset$). Then 
$\phi_{k-1}(B_k), \phi_{k-1}(D_k) \in \overline{\calA^{\prime}}$, since for $k=1$, $\phi_0(B_1) = \emptyset \in \overline{\calA^{\prime}}$ and $\phi_0(D_1) = X \in \overline{\calA^{\prime}}$, and for $k > 1$,
$B_k, D_k \in \{ C_1, \ldots, C_{k-1} \}$. It follows that $\phi_k(C_k) = F_k \in \overline{\calA^{\prime}}$, since $E_k \in \overline{\calA^{\prime}}$. Moreover, $\phi_k(C_j) = \phi_{k-1}(C_j) \in \overline{\calA^{\prime}}$ for $j \in \{ 1, \ldots, k-1 \}$. Thus, by induction, $\phi_k(C_j) \in \overline{\calA^{\prime}}$ for all $k \in \N$ and $j \in \{ 1, \ldots, k \}$. Thus $\phi^{\prime}(\calC) \subseteq \overline{\calA^{\prime}}$.

For all $k \in \N$ and $C \in \calC$, $\phi_k(C) \triangle \phi_{k-1}(C) \in \calN$. To see this, note firstly that it is trivially true for $C \not\in \calC(B_k,D_k)$. For $C = C_k$,
\begin{eqnarray*}
\phi_{k-1}(C_k) \setminus \phi_k(C_k) & = & \phi_{k-1}(C_k) \setminus F_k \\
& \subseteq & \phi_{k-1}(C_k) \setminus \big( E_k \cap \phi_{k-1}(D_k) \big) \\
& = & \phi_{k-1}(C_k) \setminus E_k,
\end{eqnarray*}
and 
\begin{eqnarray*}
\phi_k(C_k) \setminus \phi_{k-1}(C_k) & = & F_k \setminus \phi_{k-1}(C_k) \\
& \subseteq & \big( E_k \cup \phi_{k-1}(B_k) \big) \setminus \phi_{k-1}(C_k) \\
& = & E_k \setminus \phi_{k-1}(C_k).
\end{eqnarray*}
For $C \in \calC(B_k, C_k)$,
\begin{eqnarray*}
\phi_k(C) \triangle \phi_{k-1}(C) &=& \phi_{k-1}(C) \setminus F_k \\
& \subseteq & \phi_{k-1}(C_k) \setminus F_k \\
& \subseteq & \phi_{k-1}(C_k) \setminus E_k,
\end{eqnarray*}
and for $C \in \calC(C_k, D_k)$,
\begin{eqnarray*}
\phi_k(C) \triangle \phi_{k-1}(C) &=& F_k \setminus \phi_{k-1}(C) \\
& \subseteq & F_k \setminus \phi_{k-1}(C_k) \\
& \subseteq & E_k \setminus \phi_{k-1}(C_k).
\end{eqnarray*}
Hence in all cases $\phi_k(C) \triangle \phi_{k-1}(C) \subseteq E_k \triangle \phi_{k-1}(C_k) \in \calN$. It follows that $\phi^{\prime}(C)$ differs from $C$ by the addition of at most $k$ null sets and/or the removal of at most $k$ null sets, hence $\phi^{\prime}(C) \triangle C \in \calN$ for all $C \in \calC$. That is, Property~1 holds on $\calC$. This also implies $\mu(\phi^{\prime}(C)) = \mu(C)$ for all $C \in \calC$.

Define $\phi(A) := \bigcup \{ \phi^{\prime}(C) : C \in \calC, [C] \leq [A] \}$ for all $A \in \calT$. Then $\phi(C) = \phi^{\prime}(C)$ for all $C \in \calC$. Moreover, $\phi(A) \in \overline{\calA^{\prime}}$ for all $A \in \calT$. To see this, note either $[A] = [C]$ for some $C \in \calC$, which implies $\phi(A) = \phi^{\prime}(C) \in \overline{\calA^{\prime}}$, or for any $\epsilon > 0$ there exist $B, C \in \calC$ with $B \subset A \subset C$ and $\mu(C) - \mu(B) < \epsilon$, which implies $\phi(B) \subseteq \phi(A) \subseteq \phi(C)$ and $\mu(\phi(C)) - \mu(\phi(B)) < \epsilon$, and thus again $\phi(A) \in \overline{\calA^{\prime}}$. 

Note also that for any $A, B \in \calT$,
\begin{align*}
[A] = [B] &\implies \phi(A) = \phi(B) \mbox{ and } \\
[A] < [B] &\implies \phi(A) \subseteq \phi(B).
\end{align*}
Hence $\phi$ is order preserving.

Property~1 can be extended from $\calC$ to $\calT$ as follows. Consider $A \in \calT$ and $\epsilon > 0$. There exist $B, C \in \calC$ with $[B] \leq [A] \leq [C]$ and $\nu(C) - \nu(B) < \epsilon / 2$. If $[A] = [B]$ then $A \triangle B \in \calN$ and $\phi(A) = \phi^{\prime}(B)$. Moreover, $\phi^{\prime}(B) \triangle B \in \calN$. Hence $\phi(A) \triangle A \in \calN$. Similarly, if $[A] = [C]$ then $\phi(A) \triangle A \in \calN$. Suppose $[B] < [A] < [C]$, implying $B \subset A \subset C$. Then $\phi^{\prime}(B) \subseteq \phi(A) \subseteq \phi^{\prime}(C)$, since $\phi$ is order preserving, and $\nu(\phi^{\prime}(C)) - \nu(\phi^{\prime}(B)) < \epsilon / 2$. It follows that
\[
\phi(A) \triangle A \subseteq (\phi^{\prime}(B) \triangle B) \cup (\phi^{\prime}(C) \setminus \phi^{\prime}(B)) \cup (C \setminus B),
\]
and
\[
\nu^+(\phi(A) \triangle A) \leq \nu(\phi(B) \triangle B) + \nu(\phi^{\prime}(C)) - \nu(\phi^{\prime}(B)) + \nu(C) - \nu(B) < \epsilon.
\]
Letting $\epsilon \rightarrow 0$ gives $\phi(A) \triangle A \in \calN$. 

Property~2 follows from Property~1 and the fact $\phi$ is order preserving. Note
\[
\phi(A) = \phi(B) \implies [\phi(A)] = [\phi(B)] \implies [A] = [B],
\]
since $A \in [\phi(A)]$ and $B \in [\phi(B)]$. This in turn gives 
\[
[A] < [B] \implies \phi(A) \neq \phi(B) \implies \phi(A) \subset \phi(B)
\]
since it is already established that $[A] < [B] \implies \phi(A) \subseteq \phi(B)$. Moreover,
\begin{eqnarray*}
\phi(A) \subset \phi(B) &\implies& \phi(A) \neq \phi(B) \mbox{ and } \phi(B) \not\subset \phi(A) \\
&\implies& [A] \neq [B] \mbox{ and } [B] \not< [A] \\
&\implies& [A] < [B].
\end{eqnarray*}
\qed \end{proof}

The condition $\overline{\alpha(\calA^{\prime} \cup \calN)} = \alpha(\overline{\calA^{\prime}} \cup \calN)$ in the preceding lemma is implied by some stronger conditions, as follows.

\begin{lemma} \label{completion_identities}
Suppose $(X,\calA,\mu)$ is a charge space with null sets $\calN$, and $\calA^{\prime}$ is a sub-field of $\calA$. Then
\begin{enumerate}
\item $\overline{\alpha(\overline{\calA^{\prime}} \cup \calN)} = \overline{\alpha(\calA^{\prime} \cup \calN)}$,
\item $\overline{\alpha(\calA^{\prime} \cup \calN)} = \alpha(\overline{\calA^{\prime}} \cup \calN)$ if and only if $\alpha(\overline{\calA^{\prime}} \cup \calN)$ is Peano-Jordan complete.
\item $\calN \subset \overline{\calA^{\prime}} \iff \overline{\alpha(\calA^{\prime} \cup \calN)} = \overline{\calA^{\prime}} \implies \overline{\alpha(\calA^{\prime} \cup \calN)} = \alpha(\overline{\calA^{\prime}} \cup \calN)$.
\end{enumerate}
\end{lemma}

\begin{proof}
For 1, note $\overline{\alpha(\calA^{\prime} \cup \calN)} \subseteq \overline{\alpha(\overline{\calA^{\prime}} \cup \calN)}$ because $\calA^{\prime} \subseteq \overline{\calA^{\prime}}$. But also $\overline{\calA^{\prime}} \subseteq \overline{\alpha(\calA^{\prime} \cup \calN)}$ and $\calN \subseteq \overline{\alpha(\calA^{\prime} \cup \calN)}$, so $\overline{\alpha(\overline{\calA^{\prime}} \cup \calN)} \subseteq \overline{\alpha(\calA^{\prime} \cup \calN)}$.

For 2, note 1 implies $\overline{\alpha(\calA^{\prime} \cup \calN)} = \alpha(\overline{\calA^{\prime}} \cup \calN)$ if and only if $\overline{\alpha(\overline{\calA^{\prime}} \cup \calN)} = \alpha(\overline{\calA^{\prime}} \cup \calN)$.

For 3, note $\calN \subset \overline{\calA^{\prime}}$ combined with 1 gives $\overline{\alpha(\calA^{\prime} \cup \calN)} = \overline{\alpha(\overline{\calA^{\prime}} \cup \calN)} = \overline{\calA^{\prime}}$. The converse follows because $\calN \subset \overline{\alpha(\calA^{\prime} \cup \calN)}$. The last implication follows because if $\overline{\alpha(\calA^{\prime} \cup \calN)} = \overline{\calA^{\prime}}$, then $\alpha(\overline{\calA^{\prime}} \cup \calN) = \alpha(\overline{\alpha(\calA^{\prime} \cup \calN)} \cup \calN) = \overline{\alpha(\calA^{\prime} \cup \calN)}$.
\qed \end{proof}

The main reason for performing null modification operations on chains is that real-valued functions on a space $X$ generate chains of inverse images, and thus null modification of those chains can be used to construct functions with desired properties. However, the missing ingredient is a way to generate a function from a chain. The following lemma identifies a class of chains that can be placed in one-to-one correspondence with positive real-valued functions. 

\begin{lemma} \label{function_chain_equivalence}
Consider a set $X$ and a chain $\calT \subseteq \calP(X)$ of the form $\calT := \{ A_y : y \in (0,\infty) \}$. Then the following statements are logically equivalent.
\begin{enumerate}
\item There is a function $f : X \rightarrow [0,\infty)$ such that $A_y = f^{-1}(y,\infty)$ for each $y \in (0,\infty)$.
\item $\calT$ satisfies the following conditions:
\begin{enumerate}
\item $A_y = \bigcup \{ A_z \in \calT : z > y \}$ for each $y \in (0,\infty)$, and
\item $\bigcap \calT = \emptyset$, 
\end{enumerate}
\end{enumerate}
\end{lemma}

\begin{proof}
$(\implies)$ Statement~2a follows from the fact that for any $y \in (0,\infty)$,
\begin{eqnarray*}
x \in A_y &\iff& f(x) > y \\
&\iff& f(x) > z \mbox{ for some } z > y \\
&\iff& x \in A_z \mbox{ for some } z > y \\
&\iff& x \in \bigcup \{ A_z \in \calT : z > y \}
\end{eqnarray*}
Statement~2b holds because for any $x \in X$, there exists $y > f(x)$, hence $x \notin A_y$.

$(\impliedby)$ Define 
\[
f(x) := \inf \{ z \in (0,\infty) : x \notin A_z \},
\] 
noting the set on the right of this expression is not empty, since $\bigcap \calT = \emptyset$ (that is, there is no $x \in X$ with $x \in A_y$ for all $y \in (0,\infty)$). Then
\begin{eqnarray*}
x \in f^{-1}(y,\infty) & \iff & f(x) > y \\
& \iff & y < \inf \{ z \in (0,\infty) : x \notin A_z \} \\
& \iff & \exists z \in (y,\infty) \mbox{ with } x \in A_z \\
& \iff & x \in \bigcup \{ A_z \in \calT : z > y \} \\
& \iff & x \in A_y
\end{eqnarray*}
for each $y \in (0,\infty)$, giving $A_y = f^{-1}(y,\infty)$. 
\qed \end{proof}

Consequently, one can use null modification to transform $T_1$-measurable or integrable functions on the completion of a charge space to functions on the charge space itself, as described in the following lemma. Thus under the conditions of the lemma, completing the charge space has not added any new equivalence classes of functions.

\begin{lemma} \label{nullmodification_Xfn}
Suppose $(X,\calA,\mu)$ is a Peano-Jordan complete charge space with null sets $\calN$ and $\calA^{\prime}$ is a sub-field of $\calA$ such that $\overline{\alpha(\calA^{\prime} \cup \calN)} = \alpha(\overline{\calA^{\prime}} \cup \calN)$. Then for every $f \in L_0(X,\alpha(\calA^{\prime} \cup \calN),\mu)$, there exists $h \in L_0(X,\calA^{\prime},\mu)$ such that $f = h$ a.e. (with respect to $\alpha(\calA^{\prime} \cup \calN)$). Moreover, if $f \in L_p(X,\alpha(\calA^{\prime} \cup \calN),\mu)$ for some $p \in [1,\infty)$, then $h \in L_p(X,\calA^{\prime},\mu)$.
\end{lemma}

\begin{proof}
It will be sufficient to prove the theorem assuming $f : X \rightarrow \R^+$, since if the result holds for $f^+$ and $f^-$, it holds for $f$ by Lemma~\ref{T1_properties}.

By Theorem~\ref{T1measurable_characterisation}, there exists a countable set $C \subset (0,\infty)$ such that for $y \in (0,\infty) \setminus C$, $f^{-1}(y,\infty) \in \overline{\alpha(\calA^{\prime} \cup \calN)}$. Define
\begin{align*}
A_y &:= f^{-1}(y, \infty) \mbox{ for each } y \in (0,\infty) \setminus C, \mbox{ and }\\
\calT &:= \{ A_y : y \in (0,\infty) \setminus C \}. 
\end{align*}

Let $\phi : \calT \rightarrow \overline{\calA^{\prime}}$ be the order preserving map asserted in Lemma~\ref{nullmodification_X}, as it applies to $\calT$. Define
\begin{align*}
B_y &:= \phi(A_y) \mbox{ for each } y \in (0,\infty) \setminus C, \\
\calS &:= \{ B_y : y \in (0,\infty) \setminus C \}, \mbox{ and } \\
\calS_y &:= \{ B \in \calS : B \subset B_y \} \mbox{ for each } y \in (0,\infty).
\end{align*}
Then, for each $y \in (0,\infty) \setminus C$, $B_y \triangle A_y \in \calN$. Moreover, $\calS \subseteq \overline{\calA^{\prime}}$. 

Before constructing $h$, the elements of $\calS$ must be further modified to satisfy the conditions of Statement~2 of Lemma~\ref{function_chain_equivalence}, as follows. Define:
\begin{align*}
C_y &:= \bigcup \calS_y \setminus \bigcap \calS \mbox{ for each } y \in (0,\infty),  \\
\calR &:= \{ C_y : y \in (0,\infty) \}, \mbox{ and } \\
\calR_y &:= \{ C \in \calR : C \subset C_y \} \mbox{ for each } y \in (0,\infty).
\end{align*}
Then the following properties hold
\begin{enumerate}
\item $C_y \in \overline{\calA^{\prime}}$ for each $y \in (0,\infty) \setminus C$,
\item $C_y = \bigcup \calR_y$ for each $y \in (0,\infty)$,
\item $\bigcap \calR = \emptyset$, and
\item $C_y \triangle A_y \in \calN$ for each $y \in (0,\infty) \setminus C$.
\end{enumerate}

To see 1, note that since $f$ is smooth, for any $\epsilon > 0$ there is $y \in (0,\infty) \setminus C$ such that $\mu(B_y) = \mu(A_y) < \epsilon$. But $\bigcap \calS \subseteq B_y \in \overline{\calA}$ for all $y \in (0,\infty) \setminus C$, hence $\bigcap \calS \in \overline{\calA^{\prime}}$ with $\mu(\bigcap \calS) = 0$. By Lemma~\ref{countableexceptions}(2), for any $y \in (0,\infty) \setminus C$ there exist $y_1, y_2 \in (0,\infty) \setminus C$ such that $A_{y_1} \subset A_{y} \subset A_{y_2}$ and $\mu(A_{y_2} \setminus A_{y_1}) < \epsilon$. But then $B_{y_1}, B_{y_2} \in \overline{\calA^{\prime}}$ with $B_{y_1} \subseteq \bigcup \calS_y \subseteq B_{y_2}$ and $\mu(B_{y_2} \setminus B_{y_1}) < \epsilon$. Hence $\bigcup \calS_y \in \overline{\calA^{\prime}}$, and $C_y \in \overline{\calA^{\prime}}$.

Properties~2 and~3 follow immediately from the definition of $C_y$. 

To see 4, note for each $y \in (0,\infty) \setminus C$ that $C_y \triangle (\bigcup \calS_y) \in \calN$ since $\mu(\bigcap \calS) = 0$. Also note $(\bigcup \calS_y) \triangle B_y = B_y \setminus \bigcup \calS_y \in \calN$, since 
\[
\mu(B_y \setminus \bigcup \calS_y) \leq \mu(B_{y_2} \setminus B_{y_1}) < \epsilon
\]
for any $\epsilon > 0$, where $y_1, y_2$ are as defined in the proof of~1. Finally, recall $B_y \triangle A_y \in \calN$ to obtain the result.

By Lemma~\ref{function_chain_equivalence}, there is a function $h : X \rightarrow \R$ such that $h^{-1}(y,\infty) = C_y$ for all $y \in \R$. Thus $ h^{-1}(y,\infty) \in \overline{\calA^{\prime}}$ for each $y \in \R \setminus C$. By Lemma~\ref{equalityae_characterisation}(2), $h = f$ a.e. By Theorem~\ref{T1measurable_characterisation}, $h \in L_0(\N, \calA^{\prime}, \nu)$, remembering that $f$ is smooth, and hence so is $h$.  

Finally, suppose $f \in L_p(X,\alpha(\calA^{\prime} \cup \calN),\mu)$ for some $p \in [1,\infty)$. Then for any $\epsilon > 0$, there is $y \in (0, \infty)$ such that $\int | f |^p I_{| f |^{-1}(y, \infty)} d\mu_{\alpha(\calA^{\prime} \cup \calN)} < \epsilon$ by Theorem~\ref{integrable_characterisation}. It follows that
\begin{eqnarray*}
\int |h|^p I_{| h |^{-1}(y, \infty)} d \mu_{\calA^{\prime}} &=& \int |h|^p I_{| h |^{-1}(y, \infty)} d \mu_{\alpha(\calA^{\prime} \cup \calN)} \\
&=& \int |f|^p I_{| f |^{-1}(y, \infty)} d \mu_{\alpha(\calA^{\prime} \cup \calN)} \\
&<& \epsilon,
\end{eqnarray*}
where the first equality follows by Lemma~\ref{nested_fields}(4) and the second by Theorems~\ref{L1_integrable_properties}(4) and~\ref{equalityae_characterisation}(3c). Hence $h \in L_p(X,\calA^{\prime},\mu)$, by Theorem~\ref{integrable_characterisation}.
\qed \end{proof}

It is not in general true that $L_0(X,\alpha(\calA^{\prime} \cup \calN),\mu) = L_0(X,\calA^{\prime},\mu)$, as may be understood by the following example. Take any $A \in \calF$ with $0 < \nu(A) < 1$, and form the $\sigma$-field $\calA := \{ \emptyset, A, A^c, \N \}$. Let $B' \in \calN$ be any null set contained in $A^c$, and let $B = A \cup B'$. Then $I_B \in L_0(\N, \alpha(\calA \cup \calN), \nu)$ but $I_B \not\in L_0(\N, \calA, \nu)$. However, the following theorem identifies necessary and sufficient conditions to have $L_p(X,\alpha(\calA^{\prime} \cup \calN),\mu) = L_p(X,\calA^{\prime},\mu)$ and necessary and sufficient conditions to have $\calL_p(X,\alpha(\calA^{\prime} \cup \calN),\mu) \cong \calL_p(X,\calA^{\prime},\mu)$. Here the symbol `$\cong$' is used to represent an isometric vector space isomorphism. 

\begin{theorem} \label{null_sets_add_no_functions}
Suppose $(X,\calA,\mu)$ is a Peano-Jordan complete charge space with null sets $\calN$ and $\calA^{\prime}$ is a sub-field of $\calA$.  Let $p \in \{0\} \cup [1,\infty)$. Then $L_p(X,\calA^{\prime},\mu)$ is a dense subspace of $L_p(X,\alpha(\calA^{\prime} \cup \calN),\mu)$. Moreover, the following statements hold.
\begin{enumerate}
\item $L_p(X,\calA^{\prime},\mu) = L_p(X,\alpha(\calA^{\prime} \cup \calN),\mu)$ if and only if $\calN \subset \overline{\calA^{\prime}}$.
\item The following statements are logically equivalent:
\begin{enumerate}
\item $\calL_p(X,\calA^{\prime},\mu) \cong \calL_p(X,\alpha(\calA^{\prime} \cup \calN),\mu)$ with the isomorphism given by $[f]_{\calA^{\prime}} \mapsto [f]_{\alpha(\calA^{\prime} \cup \calN)}$.
\item $\overline{\alpha(\calA^{\prime} \cup \calN)} = \alpha(\overline{\calA^{\prime}} \cup \calN)$.
\end{enumerate}
\end{enumerate}
\end{theorem}

\begin{proof}
Lemma~\ref{fieldplusnull} gives that for every $A \in \alpha(\calA^{\prime} \cup \calN)$ there is $B \in \calA^{\prime}$ such that $A \triangle B \in \calN$. It follows that $I_A = I_B$ a.e. with respect to $\alpha(\calA^{\prime} \cup \calN)$, and hence that for every simple function $s$ with respect to $\alpha(\calA^{\prime} \cup \calN)$, there is a simple function $s^{\prime}$ with respect to $\calA^{\prime}$, such that $\| s - s^{\prime} \|_p = 0$ for $p \in [1,\infty)$ or $d_{\alpha(\calA^{\prime} \cup \calN)}(s,s^{\prime}) = 0$ for $p=0$. Simple functions with respect to $\alpha(\calA^{\prime} \cup \calN)$ are dense in $L_p(X,\alpha(\calA^{\prime} \cup \calN),\mu)$ by Theorem~\ref{simple_dense}, hence so are simple functions with respect to $\calA^{\prime}$, and $L_p(X,\calA^{\prime},\mu)$ is dense in $L_p(X,\alpha(\calA^{\prime} \cup \calN),\mu)$.

To show~1, it will be sufficient to show $L_p(X,\alpha(\calA^{\prime} \cup \calN),\mu) = L_p(X,\calA^{\prime},\mu)$ if and only if $\overline{\alpha(\calA^{\prime} \cup \calN)} = \overline{\calA^{\prime}}$, by Lemma~\ref{completion_identities}. First suppose $\overline{\alpha(\calA^{\prime} \cup \calN)} = \overline{\calA^{\prime}}$. Then Theorem~\ref{prop_1.8_basile2000} gives
\[
L_p(X,\alpha(\calA^{\prime} \cup \calN),\mu) = L_p(X,\overline{\alpha(\calA^{\prime} \cup \calN)},\mu) = L_p(X,\overline{\calA^{\prime}},\mu) = L_p(X,\calA^{\prime},\mu).
\]
Conversely, suppose $L_p(X,\alpha(\calA^{\prime} \cup \calN),\mu) = L_p(X,\calA^{\prime},\mu)$, and consider $A \in \overline{\alpha(\calA^{\prime} \cup \calN)}$. Then $I_A \in L_p(X,\alpha(\calA^{\prime} \cup \calN),\mu)$ by Theorem~\ref{prop_1.8_basile2000}, hence $I_A \in L_p(X,\calA^{\prime},\mu)$. By Theorem~\ref{T1measurable_characterisation}, there is $y \in (0,1)$ such that $A = I_A^{-1}(y,\infty) \in \overline{\calA^{\prime}}$.

For 2, first suppose $\overline{\alpha(\calA^{\prime} \cup \calN)} = \alpha(\overline{\calA^{\prime}} \cup \calN)$. Then Lemma~\ref{nullmodification_Xfn} implies the injective vectorspace homomorphism $[f]_{\calA^{\prime}} \mapsto [f]_{\alpha(\calA^{\prime} \cup \calN)}$ is surjective. This mapping is an isometry by Lemma~\ref{nested_fields} for $p \in [1,\infty)$, and by Corollary~\ref{nested_fields_2} for $p = 0$. Conversely, suppose $\calL_p(X,\alpha(\calA^{\prime} \cup \calN),\mu) \cong \calL_p(X,\calA^{\prime},\mu)$ with the isomorphism given by $[f]_{\calA^{\prime}} \mapsto [f]_{\alpha(\calA^{\prime} \cup \calN)}$, and consider $A \in \overline{\alpha(\calA^{\prime} \cup \calN)}$. Then there exists $f \in L_p(X,\calA^{\prime},\mu)$ such that $f = I_A$ a.e. with respect to $\overline{\alpha(\calA^{\prime} \cup \calN)}$. By Theorems~\ref{T1measurable_characterisation} and~\ref{equalityae_characterisation}, there is $y \in (0,1)$ such that $I_A^{-1}(y,\infty) \triangle f^{-1}(y,\infty) \in \calN$ and $f^{-1}(y,\infty) \in \overline{\calA^{\prime}}$. But $I_A^{-1}(y,\infty) = A$, hence $A \in \alpha(\overline{\calA^{\prime}} \cup \calN)$ by Lemma~\ref{fieldplusnull}. Thus $\overline{\alpha(\calA^{\prime} \cup \calN)} \subseteq \alpha(\overline{\calA^{\prime}} \cup \calN)$. Moreover, $\alpha(\overline{\calA^{\prime}} \cup \calN) \subseteq \overline{\alpha(\calA^{\prime} \cup \calN)}$, since $\overline{\calA^{\prime}} \subseteq \overline{\alpha(\calA^{\prime} \cup \calN)}$ and $\calN \subseteq \overline{\alpha(\calA^{\prime} \cup \calN)}$, giving $\overline{\alpha(\calA^{\prime} \cup \calN)} = \alpha(\overline{\calA^{\prime}} \cup \calN)$.
\qed \end{proof}

The requirement in several of the results in this section that $(X,\calA,\mu)$ is a Peano-Jordan complete charge space, is not as restrictive as may at first appear. By Lemma~\ref{completefield}, one may take any charge space $(X,\calA^{\prime},\mu)$ with null sets $\calN$ and form its completion $(X,\alpha(\calA^{\prime} \cup \calN),\mu)$, where now the domain of $\mu$ is extended to $\alpha(\calA^{\prime} \cup \calN)$. One can then form the charge space $(X,\overline{\alpha(\calA^{\prime} \cup \calN)},\overline{\mu})$ for use as the Peano-Jordan complete charge space in the above results. However, the above results are more general, in that they allow $\calN$ to be larger than the null sets of $(X,\calA^{\prime},\mu)$. 

This section has characterised $\calL_p(X,\calA,\mu)$ spaces for which forming the completion of the underlying charge space adds no new equivalence classes. The value of this is that in such spaces one may effectively assume $\calA$ contains the null sets of $(X,\calA,\mu)$, at least for the purpose of analysing the properties of $\calL_p(X,\calA,\mu)$. 

\bibliography{mybibfile}

\section*{Acknowledgements} 
The author is grateful to Greg Markowsky for proof reading this manuscript, providing many helpful comments, and to the Australian Research Council Centre of Excellence for Mathematical and Statistical Frontiers (CE140100049) for their support.

\end{document}